\documentclass[11pt,a4paper]{amsart}
\usepackage{color}
\usepackage{amssymb}
 \usepackage{graphicx}
\input xy
\xyoption{all}
\usepackage{tikz}
\usetikzlibrary{arrows,shapes}

\renewcommand{\epsilon}{\varepsilon}
\renewcommand{\setminus}{\smallsetminus}
\renewcommand{\emptyset}{\varnothing}
\newtheorem{theorem}{Theorem}[section]
\newtheorem{proposition}[theorem]{Proposition}
\newtheorem{corollary}[theorem]{Corollary}
\newtheorem{lemma}[theorem]{Lemma}

\theoremstyle{definition}
\newtheorem{example}[theorem]{Example}

\newtheorem{definition}[theorem]{Definition}

\newtheorem{remark}[theorem]{Remark}

\newcommand{\Z}{\mathbb Z}

\newcommand{\R}{\mathbb R}

\newcommand{\F}{\operatorname{F}}
\newcommand{\FP}{\operatorname{FP}}

\newcommand{\cohom}[3]{H^{{\raise1pt\hbox{$\scriptstyle#1$}}}(#2\>\!,#3)}
\newcommand{\tatecohom}[3] {\widehat H^{{\raise1pt\hbox{$\scriptstyle#1$}}}(#2\>\!,#3)}

\newcommand{\Cohom}[3] {H^{{\raise1pt\hbox{$\scriptstyle#1$}}}\big(#2\>\!,#3\big)}
\newcommand{\Tatecohom}[3]{\widehat H^{{\raise1pt\hbox{$\scriptstyle#1$}}}\big(#2\>\!,#3\big)}

\newcommand{\homol}[3]{H_{{\lower1pt\hbox{$\scriptstyle#1$}}}(#2\>\!,#3)}
\newcommand{\homolog}[2]{H_{{\lower1pt\hbox{$\scriptstyle#1$}}}(#2)}

\title[Finiteness properties of $2V$ and $3V$]{Cohomological finiteness properties of the Brin-Thompson-higman groups  $2V$ and $3V$}

\author{D.~ H. ~Kochloukova}

\address{Dessislava H.~Kochloukova, Department of Mathematics, University of Campinas, Cx. P. 6065,
13083-970 Campinas, SP, Brazil}\email{desi@unicamp.br}

\author{C.~Mart\'inez-P\'erez}

\address{Conchita Mart\'inez-P\'erez, Departamento de Matem\'aticas, Universidad de Zaragoza,
50009 Zaragoza, Spain} \email{conmar@unizar.es}

\author{B.~ E.~A.~Nucinkis}

\address{Brita E.~A.~Nucinkis, School of Mathematics, University of Southampton, Southampton,
SO17 1BJ, United Kingdom}\email{bean@soton.ac.uk}

\thanks{This work was partially supported by Royal Society Grant TG080182 and LMS Scheme 4 Grant 4814. The first named author was partially supported by "bolsa de produtividade em pesquisa" from CNPq, Brazil. The second named author was supported by Gobierno de Aragon and
MTM2010-19938-C03-03. }

\date{\today} 
\thanks{}

\begin{document}

\begin{abstract} We show that Brin's generalisations $2V$ and $3V$ of the Thompson-Higman group $V$ are of type $\FP_\infty$. Our methods also give a new proof that both groups are  finitely presented.
\end{abstract}

\maketitle
\thispagestyle{empty}

\section{Introduction}

In this paper we study cohomological finiteness conditions of certain generalisations of Thompson's group $V$, which is a simple, finitely presented group of homeomorphisms of the Cantor-set $C$. The finiteness conditions we consider, are the
homotopical finiteness property $\F_\infty$ for a group, which was first defined
by C.T.C.Wall, and its homological version $\FP_\infty$, which was
studied in detail in \cite{Bieribook}.  We say that a group $G$ is of type $\F_\infty$ if it admits a $K(G,1)$ with finite $k$-skeleton in all dimensions $k$. A group is of type $\FP_\infty$ if the trivial $\Z G$-module $\Z$ has a resolution with finitely generated projective $\Z G$-modules. A group is of type $\F_\infty$ if and only if it is of type $\FP_\infty$ and is finitely presented. There are, however, examples of groups of type $\FP_\infty$, which are not finitely presented \cite{Bestvina}.

\noindent In \cite{Brown2} K.S. Brown showed that Thompson's
groups $F$, $T$ and $V$ as well as some generalisations such as Higman's groups $V_{n,r}$ (see  \cite{Higman}) are  of type $\F_\infty$. The idea there is to express these groups as groups of algebra-automorphisms and let them act on a poset determined by the algebra. It is then shown that the geometric realisation of this poset yields the required finiteness properties.


\noindent
In \cite{Brin} M.~Brin defined a  group $sV$  generalising
$V$  for every natural number $s \geq 2$.  Analogously to $V$, these groups are defined  as subgroups of the homeomorphism group
of a finite Cartesian product of the Cantor-set. For each
$s$, the group $sV$ is simple, finitely presented and
contains a copy of every finite group \cite{Brin2, BHM}. It
was also shown in \cite{BL} that for $s \neq t$, $sV$ is not isomorphic to $tV.$

\noindent Our main result is the following:

\medskip\noindent{\bf Main Theorem.} {\sl Brin's groups $2V$ and $3V$ are of type $\F_\infty.$}

\medskip\noindent
The proof of the Main Theorem is split in two parts : Theorems \ref{main2} and \ref{s=3}.
We partially follow the proof of \cite{Brown2}
that $V$ has type $\F_{\infty}$. Our proof is  more intricate, as the fact that some particular complex $K_Y$ is
$t$-connected if $Y$ is sufficiently large requires more
work than in Brown's proof. 
As in \cite{Brown2} we view $sV$ as a group of algebra automorphisms and consider a
poset  $\frak A$ on which $sV$ acts.
This action has the following properties:

\begin{enumerate}
\item  Vertex stabilisers are finite.
\item  The complex $| \frak A |$ is contractible.
\item  There is a filtration  $\{ |\frak A_n| \}_{n \geq 1}$  of $sV$-subcomplexes of $| \frak A |$ such that  each complex $|\frak A_n|$ is finite modulo $sV$.
\item For $s=2$ and $s=3$ the connectivity of the pair  of complexes $(|\frak A_{n+1}|,|\frak A_n|)$ tends to infinity as $n\to\infty$.
\end{enumerate}

\noindent We then apply Brown's criterion \cite[Cor.~3.3]{Brown2} to conclude that $2V$ and $3V$ are finitely presented and
of type $\F_{\infty}$.
The key result towards the proof of our main theorem for $s = 2$ is
Theorem \ref{pushing}. Finally, in the last section,  we prove Theorem \ref{modified3} as a variation  of Theorem
\ref{pushing} and show that the  method above can be applied for  $s = 3$.


\section{Construction of the algebra and the group}

\noindent In this section we shall define the generalised Higman Algebra, also
called Cantor-Algebra, in a general setting. We then define $sV$ as a group of  automorphisms of this Algebra.

Consider a finite set $\{1,\ldots,s\}$. We call its elements colours.
Also consider  a finite set of integers $\{n_1,\ldots,n_s\}$, $n_i>1$. We call each $n_i$ the arity of the colour $i$. We begin by defining an $\Omega$-algebra $U$. For details the reader is referred to \cite{Cohn}. We say $U$ is an $\Omega$-algebra, if, for each colour $i$, the following operations are defined in $U$:
 \begin{itemize}
\item[i)] One $n_i$-ary operation $\lambda_i$:
$$\lambda_i :U^{n_i}\to U.$$
We call these  operations ascending operations, or contractions.

\item[ii)] $n_i$ 1-ary operations $\alpha^1_i,\ldots,\alpha^{n_i}_i$:
$$\alpha^j_i:U\to U.$$
We call these operations 1-ary descending operations.

 \end{itemize}
 \noindent Throughout this paper all operations act on the right.
By definition, $\Omega = \{ \lambda_i, \alpha_i^j \}_{i,j}$. In what follows it will be convenient to consider the following map, which we also call operation:  For each colour $i$, and any $v\in U$, we denote
$$v\alpha_i:=(v\alpha^1_i,v\alpha^2_i,\ldots,v\alpha^{n_i}_i).$$
Therefore $\alpha_i$ is a map
$$\alpha_i:U\to U^{n_i}.$$
We call these maps descending operations, or
expansions. Unless  otherwise stated, whenever we use the term  \lq\lq descending operation",  we refer
to one of the $\alpha_i$.

For any subset $Y$ of $U$, a simple expansion of colour $i$ of $Y$ consists of substituting some element $y\in Y$ by the $n_i$ elements of the tuple  $y\alpha_i$. And a simple contraction of colour $i$ of $Y$ is the set obtained by substituting a certain collection of $n_i$ distinct elements of $Y$, say $\{a_1,\ldots,a_{n_i}\}$, by $(a_1,\ldots,a_{n_i})\lambda_i$. We also use the word operation to refer to the effect of a simple expansion, respectively contraction on a set .

A morphism between $\Omega$-algebras is a map  commuting
with all operations in $\Omega$.
Let $\frak B_0$ be a category of $\Omega$-algebras.
An object  $U_0(X)\in\frak B_0$ is a free object in  $\frak
B_0$ with $X$ as a
{\it free basis}  if for any $S\in\frak B_0$  any mapping
$$\theta:X\to S$$
can be extended in a unique way to a morphism
$$U_0(X)\to S.$$
We also say $U_0(X)$ is free on $X$ in the category $\frak B_0$.
\noindent Following \cite[III.2]{Cohn}, we  construct the
free object on any set $X$ in the category of all
$\Omega$-algebras as follows: take the set of finite
sequences of elements of the disjoint union $\Omega\cup X$
with the $\Omega$-algebra structure defined by
juxtaposition.  Then $U_0(X)$ is  the sub $\Omega$-algebra generated by $X$.

\begin{definition} The free object constructed above is called the $\Omega$-word algebra and denoted
$W_\Omega(X)$. An {\it admissible} subset is any $Y\subset
W_\Omega(X)$, which can be obtained from $X$ by a finite
number of operations $\alpha_i$ and $\lambda_j$,  i.e. by a finite number of simple contractions or expansions.
\end{definition}


\noindent Now we   consider the variety of $\Omega$-algebras satisfying a certain set of identities.

\begin{definition} Let $\Sigma_1$ be the following set of
  laws in a countable (possibly finite) alphabet $X$.

\begin{itemize}

\item[i)] For any $u\in W_\Omega(X)$ and any colour $i$, $$u\alpha_i\lambda_i=u.$$

\item[ii)] For any colour $i$ and any $n_i$-tuple $(u_1,\ldots,u_{n_i})\in W_\Omega(X)^{n_i},$
$$(u_1,\ldots,u_{n_i})\lambda_i\alpha_i=(u_1,\ldots,u_{n_i}).$$

\end{itemize}
\end{definition}

\noindent
The variety $\frak V_1$ of $\Omega$-algebras which satisfy
the identities in $\Sigma_1$, obviously contains nontrivial
algebras. Hence it is a nontrivial variety. Therefore by
\cite[IV 3.3]{Cohn} it contains free algebras on any set
$X$. Let $U_1(X)$ be the free $\Omega$-algebra on $X$ in
$\frak V_1$. Moreover, by the proof of \cite[IV 3.1]{Cohn}
$$U_1(X)=W_\Omega(X)/\frak q_1,$$
 where $\frak q_1$ is the fully invariant congruence generated by $\Sigma_1$, i.e. the smallest equivalence set in $W_\Omega(X)\times W_\Omega(X)$ containing $\Sigma_1$, which admits any endomorphism of $W_\Omega(X)$ and is $\Omega$-closed (see \cite[IV Section 1]{Cohn}). In fact there is an epimorphism
$$\theta_1:W_\Omega(X)\to U_1(X)$$
and $\frak q_1$ corresponds precisely to $\text{Ker}(\theta_1)$.

\begin{definition}
Let $U\in\frak V_1$ and let $Y$ be a subset of $U$.
A set $Z$ obtained from $Y$ by a finite number of simple expansions is called a descendant of $Y$. In this case we denote
$$Y\leq Z.$$
 Conversely, $Y$ is called an ascendant of $Z$ and can be obtained after a finite number of simple contractions.

\end{definition}

\noindent
In what follows  we will consider $\Omega$-algebras
satisfying  some additional identities as described below.

\begin{definition}\label{validid1} Let $\Sigma$ be the set of identities
$$\Sigma=\Sigma_1\cup\{r_{ij}\mid 1\leq i<j\leq s\},$$
where $r_{ij}$ consists of certain identifications between sets of simple expansions of
$w\alpha_i$ and $w\alpha_j$ for any $w\in W_{\Omega}(X)$ which do not depend on $w$.

\noindent
Let $X$ be a set and $U(X)=U_1(X)/\frak q$ where $\frak q$ is the fully invariant congruence generated by $\Sigma$. There is an  epimorphism
$$\begin{aligned}
 \theta_2 : U_1(X)&\twoheadrightarrow U(X)\\
a_1&\mapsto \bar a_1.\\
\end{aligned}$$
Let $\theta:W_\Omega(X)\to U(X)$ be the composition of
$\theta_1$ with $\theta_2$. We say that a subset $Y$ of
$U_1(X)$ or of $U(X)$ is {\it admissible} if it is the image by $\theta_1$ or $\theta$ of an admissible subset of $W_\Omega(X)$.
We call the set of identities $\Sigma$ {\it valid} if the following condition holds:
for any admissible set $Y\subseteq U_1(X)$ we have $|Y|=|\bar Y|$, i.e.  $\theta_2$ is injective on admissible subsets.


\noindent
Let $\frak V$ be the variety of all $\Omega$-algebras which satisfy the identities in a valid $\Sigma$.  Note that $\frak V$ contains nontrivial $\Omega$-algebras, so it has free objects on every set $X$. In fact, the algebra $U(X)$ above is a free object on $X$.
\end{definition}

\begin{lemma}\label{admfree} Any admissible subset is a free
  basis in $W = U( X )$.
\end{lemma}
\begin{proof} This can be proven using the same argument as in \cite{Higman}: Let $X$ be a free basis of $W$, let $i\in\{1,\ldots,s\}$ be any colour of arity $n_i$ and
$$Y=(X\setminus\{x\})\cup\{x\alpha_i^{j}\mid 1\leq j\leq n_i\}.$$
We will show that $Y$ is a free basis of $W$.
Recall that $\frak V$ is the
variety of $\Omega$-algebras satisfying the identities $\Sigma$. Then, given any $S\in\frak V$  and any mapping
$\theta:Y\to S$, there is a unique way to obtain a map
$\theta^*:X\to S$ such that  $\theta^*(\tilde{x}) = \theta(\tilde{x})$ for $\tilde{x} \in X \setminus \{ x \}$ and $\theta^*(x) = ( \theta(x \alpha_i^1), \ldots , \theta(x \alpha_i^{n_i})) \lambda_i$. As
there is a unique $\hat\theta: W\to S$ extending $\theta^*$, the same happens with the original $\theta$.

Analogously, one proves that 
considering $n_i$ distinct elements $x_1,\ldots,x_{n_i}$ of $X$, the admissible subset
$$Y=(X\setminus\{x_1,\ldots,x_{n_i}\})\cup\{(x_1,\ldots, x_{n_i})\lambda_i\}$$
is a free basis of $W$.
\end{proof}

\begin{definition} Consider the set of $s$ colours $\{1,\ldots,s\}$, all of which have arity 2, together with the relations:
$$\Sigma:=\Sigma_1\cup\{\alpha_i^l\alpha_j^t=\alpha_j^t\alpha_i^l\mid
1\leq i \not= j\leq s; l,t=1,2 \}.$$
We call the $\Omega$-algebra $W =U(\{ x_0 \})$, defined by the $\Sigma$ above, the generalised Higman algebra on $s$ colours.
\end{definition}

\begin{remark} \label{remark00} (Geometric interpretation of the generalised Higman
  algebra). Consider the unit cube $\frak C$ of $\mathbb{R}^s$. Fix
a bijection between the set of colours $\{1,\ldots,s\}$ and
the set of hyperplanes which are parallel to the faces of
$\frak C$. To each operation $\alpha_i$ we  associate a
halving using a hyperplane parallel to the hyperplane
corresponding to $i$. In this case  we say we halve in direction $i$. Then, to each side of this halving we
 associate one of the components of $\alpha_i$:
 $\alpha_i^1$ and $\alpha_i^2$.  This association will stay fixed.
 For a sequence of 1-ary descending operations $u=\alpha_{i_1}^{r_1}\ldots\alpha_{i_t}^{r_t}$ with $r_j\in\{1,2\}$
we perform the following operations in $\frak C$: First, halve it in
direction $i_1$ and take the $r_1$-half. Repeat the process
with operation $\alpha_{i_2}^{r_2}$ for this half. At the
end, we get a subset (subparallelepiped) of $\frak C$. For simplicity we call the subparallelepipeds $s$-subcubes or simply $s$-cubes.
Note that at any stage, if $i\neq j$, the effect of
$\alpha_{i}^{r_i}\alpha_{j}^{r_j}$ equals the effect of
$\alpha_j^{r_j}\alpha_i^{r_i}$.

\bigskip
\setlength{\unitlength}{0.5cm}
\begin{picture}(8,8)(-8,0)
\thicklines
\put(0,0){\line(1,0){8}}
\put(8,0){\line(0,1){8}}
\put(8,8){\line(-1,0){8}}
\put(0,8){\line(0,-1){8}}
\put(4,0){\line(0,1){8}}
\put(4,4){\line(1,0){4}}
\put(2,0){\line(0,1){8}}
\put(4,2){\line(1,0){4}}
\put(6,4){\line(0,-1){4}}
\put(6,3){\line(1,0){2}}
\put(7,0){\line(0,1){2}}
\put(4,6){\line(1,0){4}}
\put(4,1){\line(1,0){4}}
\end{picture}

\bigskip
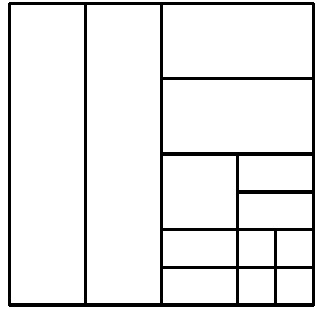
\centerline{Figure 1}

\medskip
The family of $s$-subcubes of the $s$-cube $\frak C$, which can be obtained in this
way corresponds to {\it the set $x(D)$ of descendants of $x$} in the
generalised Higman algebra $U(\{ x_0 \})$, where $x$ is an element belonging to some admissible subset.
Analogously, we may identify any admissible subset $A$ with a collection of $|A|$ $s$-cubes. In particular, the set of descendants of $A$
corresponds to the set of those subsets in the collection of $|A|$ $s$-cubes, which are obtained in the prescribed way.

\bigskip
\end{remark}


 \begin{remark}
In the following diagram  we use two different types of
carets  to visualise the two colours in the generalised Higman algebra on
$2$ colours, each of arity $2$.

\bigskip

\begin{tikzpicture}[scale=1]

  \draw[black, dashed]
    (0,0) -- (1, 1.71) -- (2,0);
  \filldraw(1,1.71) circle (0.1pt);

  \draw (0.5,0.8) node[left=8pt]{$\alpha^1_1$};
   \draw (1.5,0.8) node[right=8pt]{$\alpha^2_1$};
      \filldraw(0,0) circle (0.3pt);
        \filldraw(2,0) circle (0.3pt);

  \draw[black]
    (5,0) -- (6, 1.71) -- (7,0);
  \filldraw(6,1.71) circle (0.1pt);

  \draw (5.5,0.8) node[left=8pt]{$\alpha^1_2$};
   \draw (6.5,0.8) node[right=8pt]{$\alpha^2_2$};
      \filldraw(5,0) circle (0.3pt);
        \filldraw(7,0) circle (0.3pt);

   \end{tikzpicture}


\bigskip

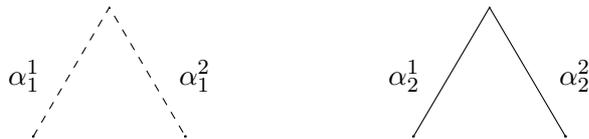
\centerline{Figure  2}

\bigskip
The first type of caret corresponds to vertical cutting and
the second one to horizontal. We view an  admissible set that
is a descendent of an element $x$ as the set of leaves of a rooted tree with
root $x$. The rooted tree is constructed by gluing one of the two
types of carets when passing to descendants.
 The following two rooted trees represent the same admissible set:

\begin{tikzpicture}[scale=0.9]

  \draw[black, dashed]
    (0,0) -- (1, 1.71) -- (2,0);
  \filldraw(1,1.71) circle (0.1pt) node[above=4pt]{$x$};

  \draw[black] (-0.8,-1.71) -- (0,0) --(0.8,-1.71);
  \draw[black] (1.2,-1.71) -- (2,0) --(2.8,-1.71);

      \filldraw (-0.8,-1.71) circle (0.3pt) node[below=4pt]{$1$};
      \filldraw (0.8,-1.71) circle (0.3pt) node[below=4pt]{$2$};
      \filldraw (1.2,-1.71) circle (0.3pt) node[below=4pt]{$3$};
      \filldraw (2.8,-1.71) circle (0.3pt) node[below=4pt]{$4$};

  \draw[black]
    (7,0) -- (8, 1.71) -- (9,0);
 \filldraw(8,1.71) circle (0.1pt) node[above=4pt]{$x$};

  \draw[black, dashed] (6.2,-1.71) -- (7,0) --(7.8,-1.71);
  \draw[black, dashed] (8.2,-1.71) -- (9,0) --(9.8,-1.71);

      \filldraw (6.2,-1.71) circle (0.3pt) node[below=4pt]{$1$};
      \filldraw (7.8,-1.71) circle (0.3pt) node[below=4pt]{$3$};
      \filldraw (8.2,-1.71) circle (0.3pt) node[below=4pt]{$2$};
      \filldraw (9.8,-1.71) circle (0.3pt) node[below=4pt]{$4$};

   \end{tikzpicture}







\bigskip

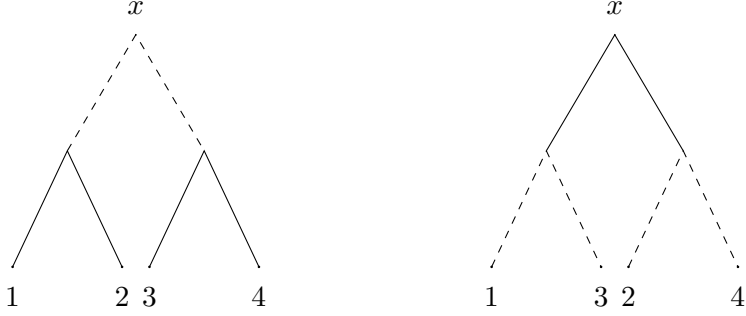
\centerline{Figure  3}

\bigskip

\noindent Considering the geometric interpretation of the generalised Higman algebra, both of the rooted trees above represent the following subdivision of the square:

\bigskip
\setlength{\unitlength}{0.4cm}
\begin{picture}(7,7)(-12,0)
\thicklines
\put(0,0){\line(0,1){6}}
\put(0,6){\line(1,0){6}}
\put(0,0){\line(1,0){6}}
\put(6,0){\line(0,1){6}}
\put(0,3){\line(1,0){6}}
\put(3,0){\line(0,1){6}}
\put(1,1){$1$}
\put(4,1){$3$}
\put(1,4){$2$}
\put(4,4){$4$}

\end{picture}
\bigskip

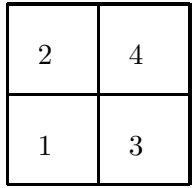
\centerline{Figure  4}

\bigskip
\end{remark}

\begin{lemma}  The generalised Higman algebra $W = U(\{ x_0 \})$ is valid.
\end{lemma}

\begin{proof}
To begin we claim that  for any pair of admissible subsets
$Y$ and $Z\subseteq U_1(\{x_0\})$, such that $Z$ is obtained from $Y$ after a simple expansion, we have  $|\bar Z|=|\bar Y|+1$. Recall that $\bar{Z}$ and $\bar{Y}$ are the images of $Z$ and $Y$ in $ U(\{ x_0 \})$.  Any admissible set in $U_1(\{x_0\})$ is a descendant of an admissible set with only one element, say $y$. So for $x = \bar{y}$ we have that  $\bar Z, \bar Y\in x(D)$, where $x(D)$ is as defined in Remark \ref{remark00}. Using the geometric interpretation of $x(D)$ as a subdivision of an $s$-cube we get the claim.

Conversely, if $Z$ is a simple contraction of $Y$ then $Y$
is a simple expansion of $Z$. Thus $|\bar Y|=|\bar
Z| + 1$.

\noindent Finally, an induction on the number of simple contractions and expansions
needed to obtain an admissible subset $\bar{Y} \subseteq
U(\{ x_0 \})$ from $\{x_0\}$ yields the result.
\end{proof}

\begin{definition}
The
 Brin-Thompson-Higman group on $W_0 = U(X)$, which we denote $G(W_0)$, is
  the group of algebra automorphisms of $W_0$ which are
  induced by a bijection $Z\to Y$ for any two admissible
  sets $Z$ and $Y$ of the same cardinality. If $W$ is the generalised Higman algebra $U(\{ x_0 \})$, then
  $G(W)$ is the Brin group on $s$ colours and is denoted $sV$.
\end{definition}

The following diagram  illustrates an element $g$ of
$2V$ sending each leaf to the leaf with the same label.

\bigskip
\begin{tikzpicture}[scale=0.9]

  \draw[black, dashed]
    (0,0) -- (1, 1.71) -- (2,0);

  \draw[black] (-0.8,-1.71) -- (0,0) --(0.8,-1.71);

      \filldraw (-0.8,-1.71) circle (0.3pt) node[below=4pt]{$1$};
      \filldraw (0.8,-1.71) circle (0.3pt) node[below=4pt]{$2$};
      \filldraw (2,0) circle (0.3pt) node[below=4pt]{$3$};

\draw[black] (4.5, 0)   node{$\longrightarrow$};
\draw[black] (4.5, 0) node[above=4pt]{$g$};

  \draw[black]
    (7,0) -- (8, 1.71) -- (9,0);

  \draw[black] (8.2,-1.71) -- (9,0) --(9.8,-1.71);

      \filldraw (7,0) circle (0.3pt) node[below=4pt]{$1$};
          \filldraw (8.2,-1.71) circle (0.3pt) node[below=4pt]{$3$};
      \filldraw (9.8,-1.71) circle (0.3pt) node[below=4pt]{$2$};

   \end{tikzpicture}



\bigskip

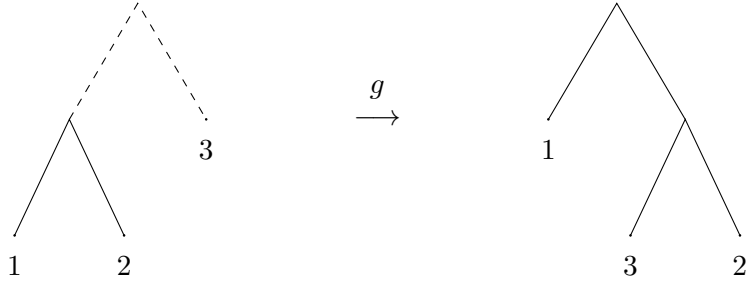
\centerline{Figure 5}

\bigskip


\begin{remark} Looking at the geometric interpretation of the generalised Higman algebra, Section 2.3  of \cite{Brin} implies that this is exactly the definition of Brin's generalisation $2V$ of $V$ as a group of self-homeomorphisms of $C\times C$, where $C$ denotes the Cantor-set. The element $g$ in Figure 5 corresponds to the following picture:

\bigskip
\begin{tikzpicture}[scale=1]

\draw[black] (0,0) --(2,0) --(2,2) -- (0,2) --(0,0);
\draw[black] (1,0) --(1,2) ;
\draw[black] (0,1) --(1,1);
\draw[black] (0.5, 1.5) node{$2$};
\draw[black] (0.5, 0.5) node{$1$};
\draw[black] (1.5, 1) node{$3$};

\draw[black] (4.5, 1)   node{$\longrightarrow$};
\draw[black] (4.5, 1) node[above=4pt]{$g$};

\draw[black] (7,0) --(9,0) --(9,2) -- (7,2) --(7,0);
\draw[black] (7,1) --(9,1) ;
\draw[black] (7,1.5) --(9,1.5);
\draw[black] (8, 0.5) node{$1$};
\draw[black] (8, 1.75) node{$2$};
\draw[black] (8, 1.25) node{$3$};

\end{tikzpicture}

\bigskip

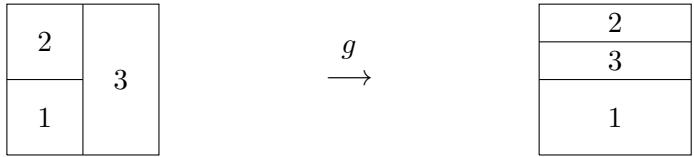
\centerline{Figure 6}

\bigskip
\noindent The equivalence of definitions for higher dimensional $sV$ follows from Section 4.1 \cite{Brin}.
If there is only one colour, then $V$ is exactly the Higman- Thompson  group as defined in \cite{Brown2}.
\end{remark}

\section{The poset of admissible subsets}

\noindent In this section we consider the Brin-Higman algebra on
$s$ colours with basis $\{ x \}$. We write $U$ for $U(\{ x \})$.

\begin{definition} \label{def00}  The set of admissible subsets is a poset
  with the order defined by $A<B$ if $B$ is a descendant of
  $A$. We denote this poset by $\frak A$ and by $|\frak A|$
  its geometric realization.
Note that any descendant and any ascendant of an
  admissible subset is also admissible.

Given admissible subsets $Y$ and $Z$ of $U$, we say that they have a unique least upper bound $T$ if $Y\leq T$  and $Z\leq  T$, and whenever $Y\leq S$ and $Z \leq S$, then $T\leq S$.
Analogously, we define the notion of greatest lower bound.
\end{definition}

\begin{lemma}\label{uniqmcd} Let $A,Y$ and $Z$ be admissible subsets with $A\leq Y$ and $A\leq Z$. Then there is a unique least upper bound of $Y$ and $Z$.
\end{lemma}
\begin{proof} Consider the geometric representation of the
  set of descendants of $A$ as subdivisions of $s$-dimensional cubes (in fact $s$-dimensional parallelepipeds but we call them cubes for simplicity) labeled by the elements of $A$, see Remark \ref{remark00}. Then the result of performing both  sets  of subdivisions corresponding to $Y$ and $Z$ yields an upper bound $T$. Clearly, for any other upper bound $S$ of $Y$ and $Z$ we have $T\leq S$.
\end{proof}

\begin{lemma}\label{step1} Let $Y$, $Y_1$ and $Z$ be admissible subsets with
 $$Y\geq Y_1\leq Z.$$
 Then there is some admissible subset $Z_1$ with
 $$Y\leq Z_1\geq Z.$$
\end{lemma}
\begin{proof} Observe that $Y$ and $Z$ are both descendants of $Y_1$. Then by Lemma \ref{uniqmcd} there exists an upper bound $Z_1$ of $Y$ and $Z$. So we have $Y\leq Z_1\geq Z$.
\end{proof}

\begin{proposition}\label{comondes} Any two admissible subsets have some upper bound.
\end{proposition}
\begin{proof} Let $Y$ and $Z$ be two admissible subsets. By definition we can obtain $Z$ from $Y$ by a finite number of expansions or contractions. Therefore we may put
$$Y\geq Y_1\leq Y_2\geq Y_3\leq\ldots \geq Y_r\leq Z.$$
By Lemma \ref{step1} we get
$$Y\leq Z_1\geq Y_2\geq Y_3\leq\ldots$$
and we may
shorten the previous chain by omitting $Y_2$ to get a chain
$$Y\leq Z_1\geq Y_3\leq\ldots$$
Thus after finitely many steps we get
$$Y\leq T\geq Z \hbox{ or } Y\geq T\leq Z$$
for some $T$. In the second case we apply Lemma \ref{step1}.
\end{proof}

Proposition \ref{comondes} has the following consequence:
for any admissible subset $A$, any element $g\in G(sV)$ can
be represented by its action in the set of descendants of $A$,
i.e. there is some $A\leq Z$ with $A\leq Zg$. To see this,
choose $Z$ to be some upper bound of $A$ and
$Ag^{-1}$. Then $A\leq Z$ and $Ag^{-1}\leq Z$, so $A\leq Zg$.

\begin{lemma}\label{contractibility} $|\frak A|$ is contractible.
\end{lemma}
\begin{proof} It is a consequence of Proposition \ref{comondes} as the poset $\frak A$ is directed.\end{proof}

\begin{remark}\label{finitestab} Observe that as in the case of $V$ considered
  in \cite{Brown2}, the stabiliser of any admissible set $Y$
  in $sV$ is finite, as it consists precisely of the permutations of the elements of $Y$.
\end{remark}

\noindent We consider the filtration of $|\frak A|$ given by
$$\frak A_n:=\{Y\in\frak A\mid|Y|\leq n\}.$$

\begin{lemma}\label{finitemodsV} Each $|\frak A_n|/sV$ is finite.
\end{lemma}
\begin{proof} For any $Y$ and $Z\in\frak A_n$ with $|Y|=|Z|$ we
may consider the element $g\in sV$ given by
  $yg=y \sigma$, where $\sigma : Y \to Z$ is a fixed
  bijection. Thus $sV$ acts transitively on the
  admissible sets of the same size.
\end{proof}

Contrary to what happens with upper bounds, it is not true in general that any two admissible subsets have some lower bound. But the existence of greatest lower bounds in some particular cases  will be crucial in the subsequent sections. To overcome this problem, we assume that our contractions are descendants of the same $A$ and consider greatest lower bounds above $A$. For simplicity we use the following notation.

\begin{definition} Let $\Lambda$ be a finite set of admissible
  sets, $A_1$ and $A_2$ be admissible sets. We write
$$ A_1 \leq \Lambda \hbox{ if for every } B \in \Lambda \hbox{ we have } A_1 \leq B$$
and
$$ \Lambda \leq A_2 \hbox{ if for every } B \in \Lambda \hbox{ we have } B \leq A_2.$$
\end{definition}


\begin{definition}\label{maxbound} Let $A$ be an admissible set and $\Omega =\{Y_0,\ldots,Y_t\}$ be a finite set of admissible sets with $A\leq \Omega$.
 Assume there exists an admissible set $M$ such that $A\leq
 M\leq \Omega$ and for any other admissible set $B$ with $A\leq B\leq \Omega$,
we have $B\leq M$. Then we call $M$ a greatest lower bound of $\Omega$ above $A$ and denote
 $M=\text{glb}_A(\Omega)$.
\end{definition}

\begin{definition}
 Let $A\leq Y$ be admissible sets and let $r\geq 0$ be an integer. We say that $A$
 involves contractions of $r$ elements of $Y$, or involves $r$ elements of $Y$ for short,  if $|Y\setminus
 A|=r$; we also say that $Y\setminus  A$ are the elements of
 $Y$ contracted in $A$. Two contractions $A_1, A_2\leq Y$ are
 said to be disjoint if the respective sets of elements of
 $Y$ contracted in $A_1$ and $A_2$ are disjoint.
 \end{definition}

\noindent
In the particular case of disjoint contractions of a certain admissible $Y$ the existence of
greatest lower bounds follows easily:

\begin{lemma} \label{remark2} Let $\Omega=\{M_0,\ldots, M_t\}$ be a set of pairwise disjoint contractions of $Y$. Then
$$\emptyset\neq\bigcap_i\{L\mid L\leq M_i\}$$
has a maximal element $M$ which we call a global greatest
lower bound for $\Omega$ and denote by $gglb(\Omega)$. In particular for any $A\leq \Omega$, $M$ is a $\text{glb}_A(\Omega)$.
Moreover
$$|\hbox{elements of }Y \hbox{ involved in }M|=\sum_{0 \leq i \leq t}|\hbox{elements of }Y \hbox{ involved in }M_i|$$

\end{lemma}

\begin{proof} We obtain $M$ by successively performing the contractions $M_i$.
\end{proof}

\begin{lemma}\label{propglb0} Let $A$ be an admissible set and $\Omega =\{Y_0,\ldots,Y_t\}$ be a finite set of  admissible sets such that $A\leq \Omega$. Then for an admissible subset $M$ we have $M=\text{glb}_A(\Omega)$ if and only if $A\leq M\leq \Omega$  and there is no expansion $N$ with $M<N$ and $N\leq \Omega$.
\end{lemma}
\begin{proof} Assume first $M=\text{glb}_A(\Omega)$. If $M< N\leq\Omega$, then $A\leq N\leq\Omega$ and therefore $N\leq M$ which is a contradiction.

Conversely, we prove that if there is no $N$ as before, then
$M$ is a greatest lower bound above $A$. Assume there is some admissible set $B$ such that  $A\leq B\leq\Omega$. Recall that by Lemma \ref{uniqmcd} there exists a unique smallest upper bound $C$ of $B$ and $M$ above $A$. Then
$$A\leq \{B,M \}\leq C\leq\Omega.$$
If $M<C$ we have a contradiction and therefore $M=C$, and thus $B\leq M$.
\end{proof}

\begin{lemma}\label{existenceglb} Let $A$ be an admissible set and $\Omega =\{Y_0,\ldots,Y_t\}$
be a finite set of
 admissible sets such that $A\leq \Omega$. Then there exists $M=\text{glb}_A(\Omega)$.
\end{lemma}
\begin{proof} Observe that the following set is finite and non-empty
$$\frak S=\{N\text{ admissible}\mid A\leq N\leq\Omega\}.$$
This means that we may choose an element $M\in\frak S$
maximal with respect to the ordering. By Lemma \ref{propglb0}, $M=\text{glb}_A(\Omega)$.
\end{proof}

For later use, we record now the following obvious consequence of the definition of greatest lower bounds and Lemma \ref{propglb0}:

\begin{lemma}\label{propglb} Let $A$ be an admissible set and $\Omega =\{Y_0,\ldots,Y_t\}$ be
a finite set of  admissible sets such that $A\leq \Omega$. Consider $A\leq B$ and a subset $\Lambda\subseteq\Omega$ such that $B\leq\Lambda$. Then
$$\text{glb}_A\Omega\leq\text{glb}_A\Lambda=\text{glb}_B\Lambda.$$
\end{lemma}

\section{Connectivity of $|K_Y|$ and proof of the main result for s=2}

Let $Y$ be any admissible subset of the Brin-Higman algebra on $s$ colours. We put

$$K_Y:=K_{<Y}=\{Z\mid Z \text{ is admissible with }Z< Y\}.$$


\noindent
Note that $K_Y$ is a poset. We also consider its geometric realisation which we denote $|K_Y|$.

 Our next objective will be to prove that in the case of two colours and $|Y|$ big enough, this complex
 $|K_Y|$ is $t$-connected.  To do this,
 we will argue as follows: firstly we will show that the
considered complex can be \lq\lq pushed down" in the sense
 that its $t$-connectedness is equivalent to the
 connectedness of a certain subcomplex $\Sigma_{4t}$
 defined in Section \ref{section4.1}.
 Then we will use an argument similar to Brown's argument in \cite{Brown2} to prove that $\Sigma_{4t}$ is $t$-connected for $|Y|$ big enough and to deduce, in the last subsection, that $2V$ is of type $\F_\infty$.

\noindent In the first subsection we shall begin with some general observations, valid for an arbitrary number $s$ of colours.

\subsection{Some general observations.}\label{section4.1}

\begin{definition} \label{def100} Denote by $C_r$ the following subposet of $K_Y$:
$$C_r:=\{A \in K_Y \mid A<Y \text{ and } A \hbox{ involves at most $r$ elements of }Y\},$$
and denote by $\Sigma_r$ the following subcomplex of $|K_Y|$:
$$\Sigma_{r}:=\{\sigma: A_t<A_{t-1}<\ldots<A_1<A_0 \mid
\sigma \in |K_Y|,  A_t\in C_r\}.$$
We denote by  $\Sigma^t_{r}$ the    $t$-skeleton of  $\Sigma_r.$
\end{definition}

To construct the pushing-procedure we will need to control the number of elements involved in the greatest lower bounds of certain sets of simple contractions of $Y$. To do that, we will use the notion of length which we define next.

\begin{definition}\label{length} Consider $A \in K_Y$. For any $i\in Y$, there  is a
  unique
$m\in A$ such that the $s$-cube labeled $m$ contains the $s$-cube labeled
$i$. Then $i$ is obtained by a certain number of successive subdivisions of $m$. We call that number the length of $i$ as descendant of $A$ and denote it by $l(A,i)$. We say
that two elements $i,j\in Y$ are glueable in $A$ if there
exists some simple contraction $Z<Y$ (of any color)
contracting exactly $i,j$ such that $A \leq Z$. Note that in that case  $l(A,i)=l(A,j)$.

We also say that $i\in Y$ is locally maximal with respect to $A$ if for any other $j\in Y$ obtained from the same $m\in A$ we have $l(A,i)\geq l(A,j)$. Clearly, in that case any other vertex which is glueable to $i$ in $A$ is also locally maximal.
\end{definition}

\noindent For example, consider the following admissible subset $A$ in the case of two colours and its descendant $Y$:

\bigskip
\begin{tikzpicture}[scale=1]

\draw[black] (1,0) --(3,0) --(3,2) -- (1,2) --(1,0);
\draw[black] (2,0) --(2,2) ;
\draw[black] (2,1) --(3,1);
\draw[black] (1.5, 1) node{$a$};
\draw[black] (2.5, 1.5) node{$b$};
\draw[black] (2.5, 0.5) node{$c$};

\draw[black] (0,1) node{$A:$};
\draw[black] (6,1) node{$Y:$};

\draw[black] (7,0) --(9,0) --(9,2) -- (7,2) --(7,0);
\draw[black] (8,0) --(8,2) ;
\draw[black] (8,1) --(9,1);
\draw[black] (7.5, 0) -- (7.5, 2);
\draw[black] (8.5, 0) -- (8.5, 1);
\draw[black] (8.5, 0.5) -- (9, 0.5);

\draw[black] (7.25, 1) node{$1$};
\draw[black] (7.75, 1) node{$2$};
\draw[black] (8.5, 1.5) node{$3$};
\draw[black] (8.25, 0.5) node{$4$};
\draw[black] (8.75, 0.75) node{$5$};
\draw[black] (8.75, 0.25) node{$6$};

\end{tikzpicture}
\bigskip

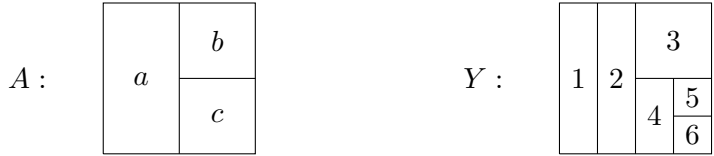
\centerline{Figure 7}

\bigskip
\noindent Here we have $l(A,5)=2$ and $6$ and $5$ are glueable. So are $1$ and $2$.
Moreover, all the elements except of 4 are locally maximal with respect to $A$.

\begin{lemma}\label{transmax} Let $A\leq B<Y$ be admissible subsets. If
  $i\in Y$ is locally maximal with respect to $A$ then it is
  also locally maximal with respect to $B$.
\end{lemma}
\begin{proof} Let $m_A\in A$, $m_B\in B$ be the elements in the respective set from which $i$ is obtained.  It suffices to note that any $j\in Y$ obtained from $m_B$  is also obtained from $m_A$.
\end{proof}

If $A\leq Y$ and we use the geometric description of $Y$ as
partitions of $s$-cubes, then the length of $i\in Y$ is related to the size of the subcube labeled $i$. If two vertices $i,j$ are glueable, then the cubes labeled $i$ and $j$ have exactly the same sizes and are neighbours. This implies that, for fixed $i$, there are at most $2s$ vertices which are glueable to $i$.
The next result implies that this bound in fact is $2(s-1)$.

\begin{lemma}\label{tec2a} Let $A\leq \{Y_0, Y_1 \} <Y$, where $Y_1$ and $Y_2$ are different, not
  disjoint, simple contractions of $Y$ of colours $a$ and $b$.  Label
  $\{1,2\}$ the vertices contracted in $Y_0$ and $\{2,3\}$
  those contracted in $Y_1$. Then the vertices labeled $1$ and $3$ are different  and
  $a\neq b$.
  \end{lemma}
\begin{proof} We use the geometric realisation of $sV$. Assume that $a=b$. As $Y_0\neq Y_1$ this would mean that the $s$-cubes labelled $1$ and $3$ are situated at opposite sides of the $s$-cube labeled 
$2$. This, however, is impossible since $\alpha^1_a$ and $\alpha_a^2$ do not commute.  In particular,  if one side of an $s$-cube can be
deleted in a contraction,
then the opposite side can not be deleted.  Therefore $a\neq b$ and
the $s$-cubes labeled $1$  and $3$
are on the sides of the $s$-cube labeled $2$ corresponding to different
directions.
In particular the $s$-cubes  labeled $1$ and $3$ are different.
 \end{proof}

 In the following definition we consider a special graph $\Gamma_A$ that will be quite useful in the next subsections.

 \begin{definition}\label{defgamma} Let $A\leq Y$ be a contraction and
  consider the coloured graph $\Gamma_A$ whose vertices are the vertices of $Y$, and with an edge of colour $a$ between vertices $i,j$ if there is a simple contraction $Z$ with $A\leq Z<Y$ which contracts $i,j$ with colour $a$.
Note that whenever $A\leq B\leq Y$ then $\Gamma_B\subseteq\Gamma_A$ and the graph $\Gamma_Y$ consists of the vertices of $Y$ with no edges.
Also, any family of simple contractions
$\Omega=\{Y_0,\ldots,Y_t\}$ of $Y$ such that $A \leq \Omega$
yields a subgraph of $\Gamma_A$
formed by the edges associated to the $Y_i$'s. We say that the family is connected if this subgraph is connected. Observe that if $\Omega$ is connected, then all the contractions $Y_i\in\Omega$ have the same length in $A$. In particular,  if the vertices involved in $Y_i$  are locally maximal with respect to $A$ then so are the vertices involved in any other $Y_j$.


\end{definition}

\subsection{Construction of the
  Pushing-procedure.} \label{section4.2}

  \noindent
From now on, we assume we have only two colours. Also  recall that both are  of arity 2.
In this subsection we  prove the following result:

\begin{theorem}\label{pushing} There exists an order reversing poset map
$$M:\{\text{Poset of simplices of }|K_Y|\}\to K_Y$$
such that for any $t$-simplex $\sigma:A_t<A_{t-1}<\ldots<A_0$ we have
$$A_t\leq M(\sigma)\in C_{4t}.$$
\end{theorem}

In the next lemma we describe certain connected components of the graph $\Gamma_A$.
Recall that for $M \in K_Y$ the vertices involved in contraction in $M$, or just involved in $M$ for short are the elements of $Y \setminus M$.

\begin{lemma}\label{tec2} Let $A\leq \{Y_0, Y_1 \} <Y$, where $Y_0$ and $Y_1$ are different, not
  disjoint, simple contractions of $Y$ such that the vertices involved in them are locally maximal with respect to some $B$ with $A\leq B\leq \{Y_0,Y_1\}$.
  Then the connected component of $\Gamma_A$ containing them is a square and for $M=\text{glb}_A(\{Y_0,Y_1 \})$, the vertices involved in $M$ are precisely those in the square. In particular, $M\in C_4$.
  \end{lemma}

\begin{proof} Label with
  $\{1,2\}$ the vertices involved in $Y_0$ and with $\{2,3\}$
  those involved  in $Y_1$.  Note that $B\leq M\leq Y_0,Y_1$ so the vertices 1,2,3 are also locally maximal respect to $M$.
 Moreover 1,2,3 are obtained from the same element $m\in M$. We shall show that the only
possibility occurring  is the picture of Figure 4, where  $m$ is
the square subdivided into 4 small squares.

Consider one of the possible chains of subdivisions of $m$
yielding
1,2,3 and let $\alpha_b$ be the first subdivision of the
chain. If $1,2,3$
were all in the same half, i.e., all descendants of the same
$m\alpha_b^r$
for a fixed $r\in\{1,2\}$ then a geometric argument proves that also
$M_1=\{m\alpha_b^1,m\alpha_b^2\} \cup (M \setminus m) \leq Y_1,Y_2$, which is impossible by the definition of greatest lower bounds.  Hence we may assume that 1,2 are partitions of
$m\alpha_b^1$ and 3 is a partition of $m\alpha_b^2$.
Moreover, by the commutativity relations, there are no more subdivisions corresponding to
colour $b$ in the path of subdivisions needed to obtain 1,2,3 from $m$. The
fact that $M\leq Y_1$ implies that the first subdivision
$\alpha_b$ can be inverted, i.e., it must be possible to
perform the successive subdivisions in such a way that the
second step consists of subdividing in direction $a$ both
halves $m\alpha_b^1$ and $m\alpha_b^2$. But again the
commutativity relations imply that we may assume that this
second subdivision using colour $a$ (i.e. subdivision
 in direction $a$)
yields precisely the line
between the rectangles 1 and 2, and that the rectangles
$1,2,3$ correspond precisely to three of the rectangles
$m\alpha_b^i\alpha_a^j$ for $i,j=1,2$.
It would be possible that the fourth rectangle were also
subdivided, but the
hypothesis that the length $l(M,1)$ is maximal implies that
this is not
the case. So the fourth is also a rectangle of the same size
which
we label 4 and therefore the rooted tree yielding $1,2,3$ from $m$ is any of the trees of Figure 3. Clearly, the associated graph in $\Gamma_A$ is a square.
\end{proof}
Observe that the previous Lemma implies that  for the contractions $Z_0$ of $\{3,4\}$ of colour $a$ and
  $Z_1$ of $\{1,4\}$ of colour $b$ we also have $A\leq M\leq
  \{Z_0,Z_1 \}$. Moreover $M=\text{glb}_A(Y_0,Y_1,Z_0)=\text{glb}_A(Y_0,Y_1,Z_0,Z_1)$.

\begin{example}\label{openbook}  If we have more than 2 colours, the obvious
corresponding version of Lemma \ref{tec2}, that two non-disjoint simple
contractions are contained in a square in $\Gamma_A$, will be false.
 Consider the following example: Suppose we have 3 colours $a,b,c$, let $A=\{m\}$ and $Y=\{1,2,3,4,5,6,7\}$ with
$$\quad 1=m\alpha_b^2\alpha_a^2\alpha_c^1,\quad 2=m\alpha_b^1\alpha_a^2\alpha_c^1,\quad 3=m\alpha_b^1\alpha_a^1\alpha_c^1,\quad 4=m\alpha_b^1\alpha_a^1\alpha_c^2,$$
$$5=m\alpha_b^1\alpha_a^2\alpha_c^2,\quad 6=m\alpha_b^2\alpha_a^2\alpha_c^2,\quad 7=m\alpha_b^2\alpha_a^1$$

\noindent Consider the following tree-diagram, where dotted lines represent halving in direction $a$, dashed lines halving in direction $b$ and normal lines halving in direction $c$.

\medskip

\begin{tikzpicture}[scale=0.9]

  \draw[black] (0,0) -- (0.5,1) --(1,0);
  \draw[black] (1.5,0) -- (2,1) --(2.5,0);
   \draw[black,] (4.5,0) -- (5,1) --(5.5,0);

  \draw[black,dotted] (0.5,1) -- (1.25,2) --(2,1);
   \draw[black,dotted] (3.5,1) -- (4.25,2) --(5,1);

  \draw[black,dashed] (1.25,2) --(2.75,3) --(4.25,2);

 \filldraw (0,0) circle (0.3pt) node[below=4pt]{$3$};
  \filldraw (1,0) circle (0.3pt) node[below=4pt]{$4$};
   \filldraw (1.5,0) circle (0.3pt) node[below=4pt]{$2$};
    \filldraw (2.5,0) circle (0.3pt) node[below=4pt]{$5$};
     \filldraw (4.5,0) circle (0.3pt) node[below=4pt]{$1$};
      \filldraw (5.5,0) circle (0.3pt) node[below=4pt]{$6$};
       \filldraw (3.5,1) circle (0.3pt) node[below=4pt]{$7$};

\end{tikzpicture}

\medskip
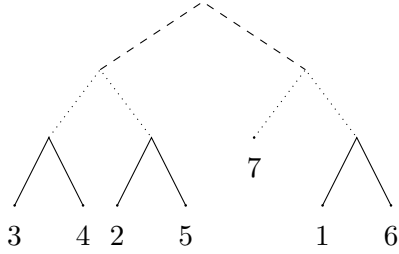
\centerline{Figure 8}
\medskip

 If we wanted all nodes of the same length, we would only have to
subdivide 7 further, for example into $m\alpha_b^2\alpha_a^1\alpha_a^1$ and
$m\alpha_b^2\alpha_a^1\alpha_a^2$.
Let $Y_0$ be the simple contraction of $Y$ of colour $b$  involving
$\{1,2 \}$ and $Y_1$ the simple contraction of $Y$ of colour $a$ involving  $\{2,3\}$.
Note that $A\leq Y_0,Y_1$ and any contraction of both $Y_0$ and $Y_1$ has to involve
contraction of either 7 elements in the first case or 8
elements in the second. 
One easily checks that (in both cases) there is no square in $\Gamma_A$ containing $Y_0$ and $Y_1$. The maximal connected component of the graph $\Gamma_A$ (in both cases)  is what will be called an open book in section 5, where we consider the case of three colours in detail. We may also represent the elements of $Y$ as subdivisions of a cube labelled $m$, the following picture ilustrates the case when $Y$ has 7 elements.

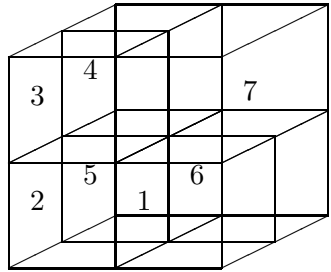
\setlength{\unitlength}{0.7cm}
\begin{picture}(6,6)(-0.1,-0.1)

\put(5,0){\line(0,1){2}}
\put(7,0){\line(0,1){2}}
\put(5,2){\line(1,0){2}}
\put(5,0){\line(1,0){2}}
\put(5,0){\line(2,1){1}}
\put(7,0){\line(2,1){1}}
\put(5,2){\line(2,1){1}}
\put(7,2){\line(2,1){1}}
\put(6,0.5){\line(0,1){2}}
\put(8,0.5){\line(0,1){2}}
\put(6,2.5){\line(1,0){2}}
\put(6,0.5){\line(1,0){2}}

\put(5,2){\line(0,1){2}}
\put(7,2){\line(0,1){2}}
\put(5,4){\line(1,0){2}}
\put(5,2){\line(1,0){2}}
\put(5,2){\line(2,1){1}}
\put(7,2){\line(2,1){1}}
\put(5,4){\line(2,1){1}}
\put(7,4){\line(2,1){1}}
\put(6,2.5){\line(0,1){2}}
\put(8,2.5){\line(0,1){2}}
\put(6,4.5){\line(1,0){2}}
\put(6,2.5){\line(1,0){2}}

\put(7,2){\line(0,1){2}}
\put(9,2){\line(0,1){2}}
\put(7,4){\line(1,0){2}}
\put(7,2){\line(1,0){2}}
\put(7,2){\line(2,1){1}}
\put(9,2){\line(2,1){1}}
\put(7,4){\line(2,1){1}}
\put(9,4){\line(2,1){1}}
\put(8,2.5){\line(0,1){2}}
\put(8,2.5){\line(1,0){2}}

\put(7,0){\line(0,1){2}}
\put(9,0){\line(0,1){2}}
\put(7,2){\line(1,0){2}}
\put(7,0){\line(1,0){2}}
\put(7,0){\line(2,1){1}}
\put(9,0){\line(2,1){1}}
\put(7,2){\line(2,1){1}}
\put(9,2){\line(2,1){1}}
\put(8,0.5){\line(0,1){2}}
\put(9,0.5){\line(0,1){2}}
\put(8,2.5){\line(1,0){2}}
\put(8,0.5){\line(1,0){2}}

\put(6,0.5){\line(0,1){2}}
\put(8,0.5){\line(0,1){2}}
\put(6,2.5){\line(1,0){2}}
\put(6,0.5){\line(1,0){2}}
\put(6,0.5){\line(2,1){1}}
\put(8,0.5){\line(2,1){1}}
\put(6,2.5){\line(2,1){1}}
\put(8,2.5){\line(2,1){1}}
\put(7,1){\line(0,1){2}}
\put(7,3){\line(1,0){2}}
\put(7,1){\line(1,0){2}}

\put(6,2.5){\line(0,1){2}}
\put(8,2.5){\line(0,1){2}}
\put(6,4.5){\line(1,0){2}}
\put(6,2.5){\line(1,0){2}}
\put(6,2.5){\line(2,1){1}}
\put(8,2.5){\line(2,1){1}}
\put(6,4.5){\line(2,1){1}}
\put(8,4.5){\line(2,1){1}}
\put(7,3){\line(0,1){2}}
\put(9,3){\line(0,1){2}}
\put(7,5){\line(1,0){2}}
\put(7,3){\line(1,0){2}}

\put(8,2.5){\line(0,1){2}}
\put(8,2.5){\line(1,0){2}}
\put(8,2.5){\line(2,1){1}}
\put(10,2.5){\line(2,1){1}}
\put(8,4.5){\line(2,1){1}}
\put(10,4.5){\line(2,1){1}}
\put(9,3){\line(0,1){2}}
\put(11,3){\line(0,1){2}}
\put(9,5){\line(1,0){2}}
\put(9,3){\line(1,0){2}}

\put(8,0.5){\line(0,1){2}}
\put(10,0.5){\line(0,1){2}}
\put(8,2.5){\line(1,0){2}}
\put(8,0.5){\line(1,0){2}}
\put(8,0.5){\line(2,1){1}}
\put(10,0.5){\line(2,1){1}}
\put(8,2.5){\line(2,1){1}}
\put(10,2.5){\line(2,1){1}}
\put(9,1){\line(0,1){2}}
\put(11,1){\line(0,1){2}}
\put(9,3){\line(1,0){2}}
\put(9,1){\line(1,0){2}}

\put(5.4,3.1){3}
\put(5.4,1.1){2}
\put(7.4,1.1){1}
\put(9.4,3.2){7}

\put(6.4,3.6){4}
\put(6.4,1.6){5}
\put(8.4,1.6){6}

\end{picture}
\smallskip

\centerline{Figure 9}

\goodbreak

\noindent Moreover, if we enlarge in a suitable
way we can easily build  examples in which the common
contraction of $Y_0,Y_1$ has to involve arbitrarily many elements of $Y$.
For example, by looking at the associated tree-diagram, we could insert another subdivision in direction
$c$ as in the figure below to obtain a $Y'$ with $13$ vertices.
As before let $Y_0$ and $Y_1$  simple contractions involving $\{1,2\}$ (with colour $b$) and $\{2,3\}$ (with colour $a$) respectively. Here, any contraction of both, $Y_0$ and $Y_1,$ would involve $13$ elements.

\bigskip

\begin{tikzpicture}[scale=0.9]

  \draw[black] (0,0) -- (0.5,1) --(1,0);
  \draw[black] (1.5,0) -- (2,1) --(2.5,0);
 \draw[black] (3,0) -- (3.5,1) --(4,0);
   \draw[black] (4.5,0) -- (5,1) --(5.5,0);
   \draw[black] (9,0) -- (9.5,1)  --(10,0);
   \draw[black] (10.5,0) -- (11,1)  --(11.5,0);

  \draw[black] (0.5,1) -- (1.25,2) --(2,1);
   \draw[black] (3.5,1) -- (4.25,2) --(5,1);
   \draw[black] (9.5,1) -- (10.25,2) --(11,1);

  \draw[black,dotted] (1.25,2) --(2.75,3) --(4.25,2);
 \draw[black,dotted] (7.25,2) --(8.75,3) --(10.25,2);

\draw[black,dashed] (2.75,3)--(5.75,4) --(8.75,3);

 \filldraw (0,0) circle (0.3pt) node[below=4pt]{$3$};
  \filldraw (1,0) circle (0.3pt) node[below=4pt]{$3'$};
   \filldraw (1.5,0) circle (0.3pt) node[below=4pt]{$4$};
    \filldraw (2.5,0) circle (0.3pt) node[below=4pt]{$4'$};
     \filldraw (3,0) circle (0.3pt) node[below=4pt]{$2$};
           \filldraw (4,0) circle (0.3pt) node[below=4pt]{$2'$};
              \filldraw (4.5,0) circle (0.3pt) node[below=4pt]{$5$};
               \filldraw (5.5,0) circle (0.3pt) node[below=4pt]{$5'$};
  \filldraw (7.25,2) circle (0.3pt) node[below=4pt]{$7$};

     \filldraw (9,0) circle (0.3pt) node[below=4pt]{$1$};
      \filldraw (10,0) circle (0.3pt) node[below=4pt]{$1'$};
                   \filldraw (10.5,0) circle (0.3pt) node[below=4pt]{$6$};
                     \filldraw (11.5,0) circle (0.3pt) node[below=4pt]{$6'$};

\end{tikzpicture}

\bigskip
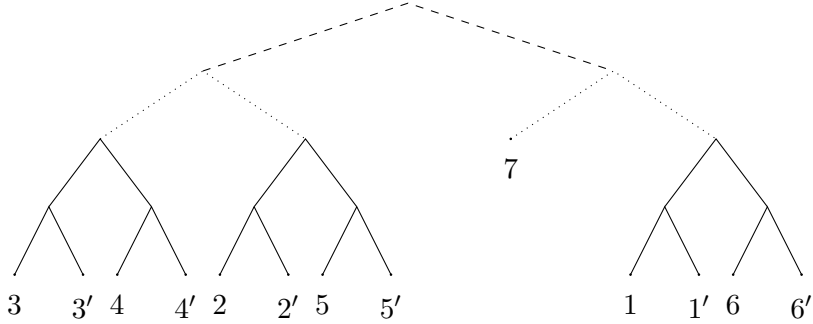
\centerline{Figure 10}

\bigskip
The effect of this in the representation of Figure 9 would be to halve each of the cubes 1, 2, 3, 4, 5 and 6 (with a plane parallel to the plane between 3 and 4) to yield the new cubes 1, 1', etc.

\end{example}

\begin{proposition}\label{general} Let
  $A \leq \Omega=\{Y_0,\ldots,Y_t\}$ where $t \geq 1$ and $Y_i$ are simple
contractions of $Y$. Assume further that there are  admissible sets
$A\leq A_t\leq A_{t-1}\leq\ldots\leq A_0$ such that
for each $i$ we have $A_i\leq Y_i$ and the elements involved in $Y_i$ are locally maximal with respect to $A_i$.
Then for $M=\text{glb}_A(\Omega)$,

 $$M\in C_{4t}.$$

\end{proposition}
\begin{proof} We may subdivide $\Omega$ into its connected components
$$\Omega=\bigcup_{i=1}^r\Omega_i.$$

For any $i \in \{ 1, \ldots , r \}$ there is  ${j_i} \in \{0,1, \ldots ,t \}$ such that $A_{j_i}\leq Y_{l_i}$ for any $Y_{l_i}\in\Omega_i$ with the elements of $Y$ contracted in $Y_{l_i}$ locally maximal with respect to $A_{j_i}$ (recall that $\Omega_i$ is connected).  Put $M_i=\text{glb}_A(\Omega_i)$.

If $\Omega_i$ contains at least two different contractions, Lemma \ref{tec2} gives that its connected component in $\Gamma_A$ is a square. In particular  $\Omega_i$ is contained in the set of four contractions representing the four sides of the square. Moreover, by the observation after Lemma \ref{tec2}, $M_i\in C_4$.

On the other hand, if all the elements of $\Omega_i$ are equal to some $Z$, then $M_i=Z\in C_2$.
Clearly, all $M_i$
are pairwise disjoint so if we put
$M=\text{glb}_A(\{M_1,\ldots,M_r\})$, then
$M=\text{glb}_A(\Omega)$ and Lemma \ref{remark2} implies
for $r \leq t$
$$|\text{vertices contracted in
}M|\leq\sum_{i=1}^r|\text{vertices contracted in }M_i | \leq 4r\leq 4t.$$
If $r = t+1$ then the elements of $\Omega$ are pairwise
disjoint and by Lemma \ref{remark2} $M \in C_{2t+2} \subseteq C_{4t}$.

\end{proof}

\noindent
Now we are ready to prove Theorem \ref{pushing}.

\begin{proof}(of Theorem \ref{pushing}) Fix any map
$$M:K_Y\to\{\text{Simple contractions of }Y\}$$
such that for any $A\in K_Y$, if $i$ is any of the elements
contracted in $M(A)$, then $i$ is locally maximal with respect to
$A$. We extend the above map $M$ to a map
$$M : \{\text{Poset of simplices of }K_Y\} \to K_Y$$ as follows: for any $t$-simplex $\sigma:A_t<A_{t-1}<\ldots<A_0$ we put
$$M(\sigma):=\text{glb}_{A_t}(M(A_t),\ldots,M(A_1),M(A_0)).$$

Proposition \ref{general} and Lemma \ref{propglb} imply that $M$ is a well defined order reversing poset map and that
$$A_t\leq M(\sigma)\in C_{4t}.$$
\end{proof}

\subsection{Construction of the null-homotopy.}

\begin{remark} Denote by $X^t$ the $t$-skeleton of a
  simplicial complex $X$. A simplicial complex $X$ is
  $t$-connected if it is $0$-connected, i.e.
  path-connected, and its $t$-th homotopy group vanishes. As
  $\pi_t(X,x_0)=[S^t,s_0;X,x_0]$, this means that every
  continuous pointed map 

$$\mu:(S^t,s_0)\buildrel \nu\over\to(X^t,x_0)\buildrel i_t\over\to(X,x_0)$$
is null-homotopic, i.e. homotopic to the constant map in $(X,x_0)$. Note, if $i_t$  is null-homotopic, then the composition $\mu=i_t\circ\nu$ will also be null-homotopic. We aim to show that $i_t$ is null-homotopic
for $|Y|$ big enough and $X=|K_Y|$.
\end{remark}

 \noindent Because of the following general result
the poset map $M$ constructed in Theorem \ref{pushing} will be useful.

\begin{lemma}\label{tec3}
Let $\frak P$ be a poset and consider an order reversing poset map
$$M:\{\text{Poset of simplices of }\frak P\}\to\frak P,$$
such that for any $\sigma: A_t
<\ldots
< A_0$, $A_t\leq M(\sigma)$ in $\frak P$. Then $M$ induces a map
$$f_t:|\frak P|^t\to |\frak P|$$
which is homotopy equivalent in $|\frak P|$ to the inclusion $i_t:|\frak P|^t\to|\frak P|$ and such that $f_t(\sigma)$ is contained in the realization of the subposet of those $B\in\frak P$ such that $M(\sigma)\leq B$.
\end{lemma}
\begin{proof} Consider the map
$$h:\{\text{Poset of simplices of }\frak P\}\to\frak P$$
such that $h(\sigma)=A_t$. Then as $h(\sigma)\leq M(\sigma)$ by a classical result in posets \cite[6.4.5]{Benson} we have $M\simeq h$. This means that $|h|\simeq|M|$.
Denote $j:\frak P\to\{\text{Poset of simplices of }\frak
P\}$ the inclusion, then $h\circ j=1_{\frak P}$. Therefore
$|1_{\frak P}|\simeq |M\circ j|$. Considering the composition
$$ f_t : |\frak P|^t\buildrel{i_t}\over\to|\frak P|\buildrel{|j|}\over\to|\{\text{Poset of simplices of }\frak P\}|\buildrel |M|\over\to|\frak P|$$
we deduce $f_t = |M|\circ |j|\circ i_t\simeq i_t$.
Finally note
that $|j|$ takes any simplex $\sigma$ to the geometric
realisation of the poset of those simplices $\delta$ such
that $\delta\subseteq\sigma$. Thus $f_t(\sigma)$ is contained in the realization of the subposet of those $B\in\frak P$ such that $M(\sigma)\leq B$.
\end{proof}

\noindent
As a corollary of  Definition \ref{def100}, Theorem \ref{pushing} and
Lemma \ref{tec3}
we obtain the following result.

\begin{proposition} \label{oxford1}
For any $t$ there is a map
$$f_t:|K_Y|^t\to |K_Y|$$
which is homotopy equivalent to the inclusion $i_t:|K_Y|^t\to|K_Y|$
and such that $f_t(\sigma)\subseteq\Sigma^t_{4t}$.
\end{proposition}

\begin{lemma} \label{important32}
For any  fixed $r,t$ there exists a function $\nu_r(t)$ such that if $|Y|\geq\nu_r(t),$ the inclusion of $\Sigma^t_r$ in $|K_Y|$ is null-homotopic.
\end{lemma}
\begin{proof} We adapt Brown's argument in \cite[4.20]{Brown2} to our context. For $|Y|$ big enough we will  construct,  by induction on $t$,  a
null-homotopy $$F_t:\Sigma^t_r\times I\to|K_Y|$$
such that $F_t ( - , 0)$ is the
identity map and $F_t ( -, 1)$ is the
constant map  sending everything  to the point $a \in K_Y$. More
precisely, we do the following: we show that there are
functions $\nu_r(t)$, $\mu_r(t)$ such that for
$|Y|\geq\nu_r(t)$ there is a homotopy $F_t$ as before, such
that for any $t$-simplex $\sigma\in\Sigma^t_r$,
$F_t(\sigma\times I)\subseteq \widehat{\Sigma}_{\mu_r(t)},$
where $\widehat{\Sigma}_s$ is the set of subcomplexes $T$ of
$\Sigma_s$ such that the union of all
elements of $Y$ that are contracted in the vertex of some
simplex of $T$ has at most $s$ elements.

{\it  The case $t=0$:} We choose any simple
contraction $a$ of $Y$. Hence it involves 2 vertices,
i.e. elements of $Y$. 
Let $A$ be a point of $\Sigma_r^0$ i.e. $A$
 is a
contraction of $Y$ involving at most $r$ vertices. Now, if
$|Y|\geq r+4$, we may choose a set of $2$ vertices  disjoint
to both those contracted in $A$ and those contracted in
$a$. Let $b_0$ be a simple contraction of any colour of $Y$ corresponding
to these two vertices. Then
$$A \geq gglb(A, b_0) \leq b_0 \geq gglb(b_0,a) \leq a$$
is a path linking $A$ with $a$ in $\widehat{\Sigma}^0_{r+4}$.
 Therefore we get the claim with
$$\nu_r(0)=r+4,$$
$$\mu_r(0)=r+4.$$

{\it Induction step}: We assume there is a null-homotopy
$F_{t-1}:\Sigma^{t-1}_r\times I\to K_Y$. We want to extend
$F_{t-1}$ to $F_t$.
 Let $\sigma: A_t< A_{t-1}< ... <A_0$ be a $t$-simplex in $\Sigma^t_r$. For any face $\tau$ of $\sigma$  of dimension $t-1$  we have
$F_{t-1}(\tau\times I)\subseteq \widehat{\Sigma}_{\mu_r(t-1)}.$
 This means that if we denote $\delta\sigma=\cup_{i=1}^{t+1}\tau_i$, then
$$\Delta:=F_{t-1}(\delta\sigma\times I)=\cup
F_{t-1}(\tau_i\times I)\subseteq
\widehat{\Sigma}_{(t+1)\mu_r(t-1)}.$$
Now, if $|Y|\geq 2+(t+1)\mu_r(t-1)$  there are
at least $2$ vertices of $Y$ not involved in any contraction
in $F_{t-1}(\delta\sigma\times I)$. Let $b$ be a simple
contraction of any colour of $Y$ contracting these 2 vertices.

We claim that the homotopy $F_{t-1}$ can be extended to $F_t:\sigma\times I\to|K_Y|$ with
$$F_t(\sigma\times I)\subseteq \widehat{\Sigma}_{2+(t+1)\mu_r(t-1)}.$$
As $b$ is a contraction of $Y$ disjoint to all contractions
$A$ of $Y$ such that  $A\in F_{t-1}(\delta\sigma\times I)$, we may consider the global greatest lower bound of $b$ and $A$ which we denote $\text{gglb}(A,b)$. Note that this is just the result of contracting in $A$ those elements which are contracted in $b$.
Analogously we denote by $\text{gglb}(\Delta,b)$ the subcomplex given by $\text{gglb}(A,b)$ for all $A \in \Delta.$ The same notation is also used for simplices in $\Delta$.
Also note that for all $A \in \Delta$, $\text{gglb}(A,b) \leq b$ and we can always form the cone with base $\text{gglb}(\Delta,b)$ and vertex $b$.

We claim that the homotopy $F_{t}(\sigma\times I)$ can be  built up by gluing:
\begin{itemize}
\item[i)] the cylinder given by $\Delta$ and $\text{gglb}(\Delta,b)$

\item[ii)] the cone formed by $\text{gglb}(\Delta,b)$ and $b$.
\end{itemize}

Note that for any $l$-simplex $\tau:A_l<A_{l-1}<\ldots<A_0$ lying in $\Delta$ then the following $l+1$-simplices:
$$\text{gglb}(A_l,b)<\text{gglb}(A_{l-1},b)<\ldots<\text{gglb}(A_i,b)<A_i<A_{i-1}<\ldots<A_0$$
for $i=l,\ldots,0$ fill up the cylinder formed by  $\tau$ and $\text{gglb}(\tau,b)$ (recall that $\text{gglb}(\tau,b)$ is given by $\text{gglb}(A_l,b)<\text{gglb}(A_{l-1},b)<\ldots<\text{gglb}(A_0,b)$).

Furthermore, the cone formed by  $\text{gglb}(\tau,b)$ and $b$ is also filled up via the $t+1$-simplex
$$\text{gglb}(A_l,b)<\text{gglb}(A_{l-1},b)<\ldots<\text{gglb}(A_0,b)<b.$$

We shall now explain how the above constructions yield the extension of the homotopy :

\noindent
(1) Consider the cylinder with base the simplex $\sigma$ and
top the simplex
$gglb(\sigma, b)$ and glue to the cylinder
 the cone with base $gglb(\sigma,b)$ and apex
$b$.

\bigskip
\setlength{\unitlength}{0.7cm}
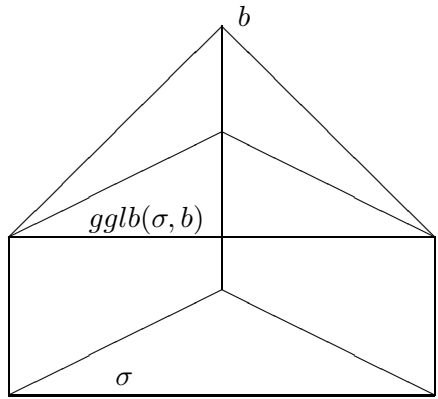
\begin{picture}(12,8)(-0.5,-0.5)
\put(0,0){\line(1,0){8}}
\put(8,0){\line(-2,1){4}}
\put(0,0){\line(2,1){4}}
\put(0,3){\line(1,0){8}}
\put(8,3){\line(-2,1){4}}
\put(0,3){\line(2,1){4}}
\put(0,0){\line(0,1){3}}
\put(8,0){\line(0,1){3}}
\put(4,2){\line(0,1){3}}
\put(0,3){\line(1,1){4}}
\put(8,3){\line(-1,1){4}}
\put(4,5){\line(0,1){2}}
\put(4.3,7){$b$}
\put(1.5,3.2){$gglb(\sigma,b)$}
\put(2,0.2){$\sigma$}
\end{picture}

\bigskip

\centerline{Figure 11}

\bigskip

Let $\sigma \cup \widetilde{\Sigma}$ be the boundary of Figure 11. Then
$\sigma$ is homotopic to $\widetilde{\Sigma}$ via a homotopy, see
 Figure 11, fixing
$\partial \sigma$.

\noindent
(2) The following picture illustrates the homotopy $F_{t-1}$
squeezing $\partial \sigma$ to the point $a$.

\bigskip
\setlength{\unitlength}{0.7cm}
\begin{picture}(8,3)(-0.5,-0.5)
\put(0,0){\line(1,0){8}}
\put(8,0){\line(-2,1){4}}
\put(0,0){\line(2,1){4}}
\put(4,1){$a$}
\put(3.7,0.8){$.$}
\put(2,0.2){$\Delta$}
\end{picture}

\bigskip

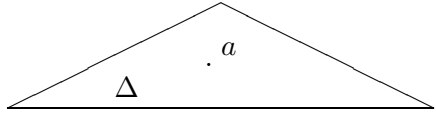
\centerline{Figure 12}

\bigskip\noindent
(3) Consider the cylinder with bottom $\Delta$ and top
$gglb(\Delta,b)$ and glue to it the cone with bottom $gglb(\Delta,b)$
and vertex $b$.

\bigskip
\setlength{\unitlength}{0.7cm}
\begin{picture}(12,8)(-0.5,-0.5)
\put(0,0){\line(1,0){8}}
\put(8,0){\line(-2,1){4}}
\put(0,0){\line(2,1){4}}
\put(1,3){\line(1,0){8}}
\put(9,3){\line(-2,1){4}}
\put(1,3){\line(2,1){4}}
\put(0,0){\line(1,3){1}}
\put(8,0){\line(1,3){1}}
\put(4,2){\line(1,3){1}}
\put(1,3){\line(3,4){3}}
\put(9,3){\line(-5,4){5}}
\put(5,5){\line(-1,2){1}}

\put(4.3,7){$b$}
\put(2.5,3.1){$gglb(\Delta,b)$}
\put(2,0.2){$\Delta$}
\put(4,1){$a$}
\put(3.7,0.8){$.$}
\put(5,4){$gglb(a,b)$}
\put(4.7,3.8){$.$}
\end{picture}

\bigskip

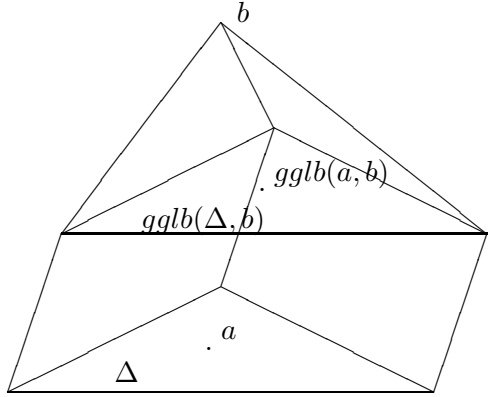
\centerline{Figure 13}

\bigskip

Note that $\Delta \cup \widetilde{\Sigma}$ is the boundary of Figure 13.
Thus $\widetilde{\Sigma}$ and $\Delta$ are homotopy equivalent via a
homotopy, see Figure 13,  fixing $\partial \Delta = \partial \sigma$. Set $\mu_r(t)=2+(t+1)\mu_r(t-1)$.
Then by (1) and (3) $\sigma$ and $\Delta$ are homotopy
equivalent via a homotopy, inside $\widehat{\Sigma}_{\mu_r(t)}$, which fixes $\partial \sigma$. This
completes the proof of the fact that $F_{t-1}$ is extendable
to a homotopy $F_t$ (inside $\widehat{\Sigma}_{\mu_r(t)}$) that
contracts $\sigma$ to the point $a$. Therefore the inductive step is proven for
$$\nu_r(t)=\mu_r(t)=2+(t+1)\mu_r(t-1).$$

 \end{proof}

\begin{theorem} \label{important31}
There exists a function $\alpha(t)$ such that if $|Y|\geq\alpha(t),$  the inclusion of  $|K_Y|^t$ in $|K_Y|$ is null-homotopic.
\end{theorem}
\begin{proof}
Consider the homotopy equivalent maps
$i_t,f_t:|K_Y|^t\to|K_Y|$ given by Proposition \ref{oxford1}. Since the image of
$f_t$ is  contained in $\Sigma^t_{4t}$, $f_t$
factors through the inclusion of $\Sigma^t_{4t}$ in $K_Y$. But we have just proven that this last inclusion is null-homotopic whenever $|Y|\geq\nu_{4t}(t)$ and therefore in that case $f_t$ and $i_t$ are also null-homotopic.
Therefore it suffices to set $\alpha(t):=\nu_{4t}(t)$.
\end{proof}

\begin{corollary} \label{connectivity1} There exists a function $\alpha(t)$ such that if $|Y|\geq\alpha(t),$  $K_Y$ is $t$-connected.
\end{corollary}

\subsection{Finiteness properties of $2V$}\hfill\break

\noindent
Now we are ready to prove that the group $2V$ is of type
$\FP_\infty$. To do that, we will verify the conditions
of \cite[Cor.~3.3]{Brown2}   with respect to the
complex $|\frak A|$ defined in Definition \ref{def00}. 
As before, consider the filtration of $|\frak A|$ given by
$$\frak A_n:=\{Y\in\frak A\mid|Y|\leq n\}.$$
Lemmas \ref{contractibility} and \ref{finitemodsV}  and Remark \ref{finitestab} imply that all that remains is to prove the following theorem.

\begin{theorem}\label{conectivity} The connectivity of the pair of complexes $(|\frak A_{n+1}|,|\frak A_n|)$ tends to infinity as $n\to\infty$.
\end{theorem}
\begin{proof} We use the same argument as in
  \cite[4.17]{Brown2} i.e. note that $|\frak A_{n+1}|$ is obtained from $|\frak A_n|$ by gluing cones with base $K_{Y}$ and top $Y$ for every $Y \in {\frak A_{n+1}} \setminus {\frak A_n}$. By Corollary \ref{connectivity1}, if $n+1 \geq \alpha(t)$ we have that $K_Y$ is $t$-connected, hence $(|\frak A_{n+1}|,|\frak A_n|)$ is $t$-connected.
\end{proof}

\begin{theorem}\label{main2} The Brin group on 2 colours each of arity 2 i.e. $2V$, is of type $\F_\infty$.
\end{theorem}
\begin{proof} By Lemmas \ref{contractibility} and \ref{finitemodsV}, Remark \ref{finitestab} and Theorem \ref{conectivity} we may apply \cite[Cor.~3.3]{Brown2}.
\end{proof}

\begin{remark} As a by-product, we get by
  \cite[Cor.~3.3]{Brown2}  a new proof
of  the fact that $2V$ is finitely presented. This was first proved in \cite{Brin2}, where an explicit finite presentation was constructed.
\end{remark}

\section{The case $s = 3$}

\noindent
In this section  we consider  Brin's group $sV$ for $s = 3$.
Our objective is to show that $3V$ is of type $\F_{\infty}$ by
adapting the construction of the function $M$
of Lemma \ref{tec3} to the case $s = 3$. In particular we show that Theorem \ref{pushing} holds with $M \in C_{8t}$. This immediately leads to a modification of Proposition \ref{oxford1} that $f_t(\sigma) \in \Sigma^t_{8t}.$ The rest of the proof will be analogous to the previous case.

\noindent As before, we fix a $Y$ and  prove that $K_Y$ is
$t$-connected if $|Y|$ is sufficiently large. For
$A < Y$ we consider the coloured graph $\Gamma_A$ as in Definition \ref{defgamma}. This
time the graph is embedded in 3 dimensional real space and
the three possible colours $\{ a,b,c \}$ correspond to the
axes of the standard coordinate system of $\R^3$. For any subgraph $\Delta\subseteq\Gamma_A$ we put
$$\text{glb}_A(\Delta):=\text{glb}_A\{\text{Simple contractions associated to the edges of }\Delta\}.$$

Consider a connected component $\Delta$ of $\Gamma_A$. The vertices of $\Delta$ correspond, via the geometric realisation of $3V$,
to subparallelepipeds of the unit cube $I$, all of the same shape
and size.  For simplicity,  we draw them as cubes and call them subcubes.
 Let $i$ be an element (i.e. a vertex) of $\Delta$.  By some abuse of notation we shall also label by $i$ the subcube corresponding to the element $i$ of $\Delta.$

\noindent
We claim that the vertices of $\Delta$ are inside a stack of 8
 subcubes, see Figure 14. Obviously one of these subcubes
 corresponds to $i$. 
Observe that we do not claim that  all the subcubes in the
stack correspond 
to elements of $Y$, only that $\Delta$ is a set consisting
of some of 
the subcubes in the stack.
To see that the claim holds,  let $i$ be $[\alpha_1, \alpha_2] \times [\beta_1 , \beta_2] \times [\gamma_1, \gamma_2]$.
The interval $ A_0 = [\alpha_1, \alpha_2]$ comes from a set of binary subdivisions of $[0,1]$. The left descendant of an interval $[x,y]$ is $[x, (x+y)/2]$ and the right descendant of $[x,y]$ is $[(x+y)/ 2, y]$. Then $A_0$ is a descendant of some interval $J_A$ that is subdivided into $A_0$ and $A_1$ in the binary subdivision. Recall, see for example Lemma \ref{tec2a}, that each cube in a connected set can only have one neighbour of each colour/direction. Define  $B_1$ and $C_1$ analogously. Then the  cubes in the stack containing $\Delta$ are precisely the cubes $A_i \times B_j \times C_k$, where $i,j,k \in \{ 0,1 \}$.


\bigskip
\setlength{\unitlength}{0.8cm}
\begin{picture}(6,6)(-0.1,-0.1)

\put(5,0){\line(0,1){2}}
\put(7,0){\line(0,1){2}}
\put(5,2){\line(1,0){2}}
\put(5,0){\line(1,0){2}}
\put(5,0){\line(2,1){1}}
\put(7,0){\line(2,1){1}}
\put(5,2){\line(2,1){1}}
\put(7,2){\line(2,1){1}}
\put(6,0.5){\line(0,1){2}}
\put(8,0.5){\line(0,1){2}}
\put(6,2.5){\line(1,0){2}}
\put(6,0.5){\line(1,0){2}}

\put(5,2){\line(0,1){2}}
\put(7,2){\line(0,1){2}}
\put(5,4){\line(1,0){2}}
\put(5,2){\line(1,0){2}}
\put(5,2){\line(2,1){1}}
\put(7,2){\line(2,1){1}}
\put(5,4){\line(2,1){1}}
\put(7,4){\line(2,1){1}}
\put(6,2.5){\line(0,1){2}}
\put(8,2.5){\line(0,1){2}}
\put(6,4.5){\line(1,0){2}}
\put(6,2.5){\line(1,0){2}}

\put(7,2){\line(0,1){2}}
\put(9,2){\line(0,1){2}}
\put(7,4){\line(1,0){2}}
\put(7,2){\line(1,0){2}}
\put(7,2){\line(2,1){1}}
\put(9,2){\line(2,1){1}}
\put(7,4){\line(2,1){1}}
\put(9,4){\line(2,1){1}}
\put(8,2.5){\line(0,1){2}}
\put(10,2.5){\line(0,1){2}}
\put(8,4.5){\line(1,0){2}}
\put(8,2.5){\line(1,0){2}}

\put(7,0){\line(0,1){2}}
\put(9,0){\line(0,1){2}}
\put(7,2){\line(1,0){2}}
\put(7,0){\line(1,0){2}}
\put(7,0){\line(2,1){1}}
\put(9,0){\line(2,1){1}}
\put(7,2){\line(2,1){1}}
\put(9,2){\line(2,1){1}}
\put(8,0.5){\line(0,1){2}}
\put(9,0.5){\line(0,1){2}}
\put(8,2.5){\line(1,0){2}}
\put(8,0.5){\line(1,0){2}}

\put(6,0.5){\line(0,1){2}}
\put(8,0.5){\line(0,1){2}}
\put(6,2.5){\line(1,0){2}}
\put(6,0.5){\line(1,0){2}}
\put(6,0.5){\line(2,1){1}}
\put(8,0.5){\line(2,1){1}}
\put(6,2.5){\line(2,1){1}}
\put(8,2.5){\line(2,1){1}}
\put(7,1){\line(0,1){2}}
\put(10,1){\line(0,1){2}}
\put(7,3){\line(1,0){2}}
\put(7,1){\line(1,0){2}}

\put(6,2.5){\line(0,1){2}}
\put(8,2.5){\line(0,1){2}}
\put(6,4.5){\line(1,0){2}}
\put(6,2.5){\line(1,0){2}}
\put(6,2.5){\line(2,1){1}}
\put(8,2.5){\line(2,1){1}}
\put(6,4.5){\line(2,1){1}}
\put(8,4.5){\line(2,1){1}}
\put(7,3){\line(0,1){2}}
\put(9,3){\line(0,1){2}}
\put(7,5){\line(1,0){2}}
\put(7,3){\line(1,0){2}}

\put(8,2.5){\line(0,1){2}}
\put(10,2.5){\line(0,1){2}}
\put(8,4.5){\line(1,0){2}}
\put(8,2.5){\line(1,0){2}}
\put(8,2.5){\line(2,1){1}}
\put(10,2.5){\line(2,1){1}}
\put(8,4.5){\line(2,1){1}}
\put(10,4.5){\line(2,1){1}}
\put(9,3){\line(0,1){2}}
\put(11,3){\line(0,1){2}}
\put(9,5){\line(1,0){2}}
\put(9,3){\line(1,0){2}}

\put(8,0.5){\line(0,1){2}}
\put(10,0.5){\line(0,1){2}}
\put(8,2.5){\line(1,0){2}}
\put(8,0.5){\line(1,0){2}}
\put(8,0.5){\line(2,1){1}}
\put(10,0.5){\line(2,1){1}}
\put(8,2.5){\line(2,1){1}}
\put(10,2.5){\line(2,1){1}}
\put(9,1){\line(0,1){2}}
\put(11,1){\line(0,1){2}}
\put(9,3){\line(1,0){2}}
\put(9,1){\line(1,0){2}}

\end{picture}

\bigskip
\centerline{A stack of 8 cubes}

\medskip
\centerline{Figure 14}

\bigskip

\noindent
For a connected component $\Delta$ of $\Gamma_A$ we define
the enveloping stack of $\Delta$ to be the smallest set
$U(\Delta)$ of some subcubes from the 8 cube stack defined
above such that  $U(\Delta)$ contains all  $i \in\Delta$, and the union of the elements of $U(\Delta)$ is a cube.

\noindent
Note that if one of the vertices of $\Delta$ is
locally maximal with respect to some $C < Y$ such that $A
\leq C$ then every vertex of $\Delta$ is locally maximal
with respect to $C$.  This leads to the following definition.

\begin{definition}\label{starcon}
A connected component $\Delta$ of $\Gamma_A$ is called
$*$-connected 
if there is some admissible set $C$ such that  $A \leq C<Y$
and
 every vertex of $\Delta$ is locally maximal with respect to $C$.
\end{definition}

\noindent
The
following diagram  exhibits possible $*$-connected components of
the graph $\Gamma_A$ for $A < Y$. Note that parallel edges
are labeled by the same colour.

\bigskip
\setlength{\unitlength}{0.7cm}
\begin{picture}(22,4)(-0.1,-0.5)
\put(1.5,0){\line(0,1){2}}
\put(1,1){$a$}
\put(3,0){\line(0,1){2}}
\put(2.5,1){$a$}
\put(5,0){\line(0,1){2}}
\put(4.5,1){$a$}
\put(3,2){\line(1,0){2}}
\put(3.8,1.5){$b$}
\put(3,0){\line(1,0){2}}
\put(3.8,-0.5){$b$}


\put(6.6, 1){$a$}
\put(7.8, -0.5){$b$}
\put(7.5, 0.6){$c$}
\put(7,0){\line(0,1){2}}
\put(7,0){\line(1,0){2}}
\put(7,0){\line(2,1){1}}
\put(9,0){\line(2,1){1}}
\put(7,2){\line(2,1){1}}
\put(8,0.5){\line(0,1){2}}
\put(8,0.5){\line(1,0){2}}

\put(11.6, 1){$a$}
\put(12.8, -0.5){$b$}
\put(12.5, 0.6){$c$}
\put(12,0){\line(0,1){2}}
\put(14,0){\line(0,1){2}}
\put(12,2){\line(1,0){2}}
\put(12,0){\line(1,0){2}}
\put(12,0){\line(2,1){1}}
\put(14,0){\line(2,1){1}}
\put(12,2){\line(2,1){1}}
\put(14,2){\line(2,1){1}}

\put(13,0.5){\line(0,1){2}}
\put(15,0.5){\line(0,1){2}}
\put(13,2.5){\line(1,0){2}}
\put(13,0.5){\line(1,0){2}}
\end{picture}

\bigskip

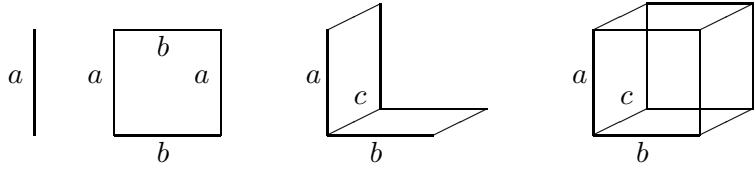
\centerline{Figure 15}

\bigskip
\noindent
We call the graphs in Figure 15 an edge, a square, an open book and a cube respectively.

\begin{lemma}\label{graphs}  Let $\Delta$ be a $*$--connected component of $\Gamma_A$. Then, up to changing the colours, $\Delta$ is one of the graphs in
  Figure 15. Moreover, if $\Delta$ is not an open book, then for
  $M=\text{glb}_A(\Delta)$  the vertices involved in $M$ lie inside $\Delta$. In particular, $M\in C_8$.
\end{lemma}

\begin{proof} We argue as in  Lemma \ref{tec2}. We consider the element $m\in M$ which yields $\Delta$, i.e. the vertices of $\Delta$ are obtained from $m$ by the halving operations. Observe that $M=\{m\}\cup(M\cap Y)$.
Consider the geometric realisation of $M$. Then $m$ is a subcube of the unitary cube and the enveloping stack $U(\Delta)$  lies inside $m$.
Since $M<Y$ we may choose some simple expansion $M<M_1\leq Y$ of colour $a$, say.
The expansion $M<M_1$ corresponds to halving the cube $m$ by
a hyperplane of direction (i.e. colour) $a$. Furthermore,
this halving also yields  a halving of the enveloping stack
$U(\Delta)$. In other words, not all the vertices of
$\Delta$ are in the same half of $m$, as that would mean
that $M=M_1$. Moreover, as $\Delta$ is connected, this
halving can be inverted, by using the commutativity
relations, to give a simple contraction of $Y$ of direction $a$. If $M_1=Y$, then $\Delta$ is an edge and $M\in C_2$.

\noindent
Note, that since the halving operation of $m$ in direction $a$ halves $U(\Delta)$, we have an edge $e$ in $\Delta$  with label $a$  and vertices $i,j$.  In particular,  the elements $i$ and $j$ represent neighbouring cubes  in $U(\Delta)$,  one  contained in  $m \alpha_a^1$ and the other  in $m \alpha_a^2$.
Since $e \in \Gamma_A$ there is a contraction of $Y$
contracting precisely $i$ and $j$. This implies that in the
process of obtaining $Y$ from $M$ via halving operations on
$m$, there is another chain of halving operations starting
with halving in a direction different from $a$,  say $b$. Hence, by the commutativity relations,
there exists $M_2$ with $M_1<M_2\leq Y$
 such that $M_2$ consists of halving  both $m\alpha_a^1$ and $m\alpha_a^2$ in direction $b$.
Clearly, this allows inversion and therefore the above procedure for $a$ can also be applied to $b$ . After performing these two subdivisions we get a stack $S$ of four cubes. Moreover, we may assume that there are vertices of $\Delta$ lying in at least three of those four cubes. Otherwise $\Delta$ would be either disconnected or $M \neq glb_A(\Delta)$. Note also that, to obtain $\Delta$, only halving of those four cubes in a direction $c$ different from directions $a$ and $b$ is possible. So it remains  to consider the  three possibilities below. Recall,  we are assuming that $\Delta$ is $*$-connected.

\noindent (1)
If none of the cubes is halved, then $M_2=Y$, $\Delta$ is a square and $M\in C_4$.

\noindent (2)
Suppose all four cubes are halved at least once. Then  the rooted tree representing the way $\Delta$ is obtained  from $m$,
starts as the first tree in Figure 16 below. In this case we may use the commutativity relations to get a rooted tree with halving in direction $c$ at the beginning. Therefore, the assumptions that $\Delta$ is connected and that $M=\text{glb}_A(\Delta)$  imply that in fact there is only one halving in direction $c$.  In particular, the rooted tree is precisely the first tree in Figure 16. Thus $\Delta$ is a cube, $m$ yields the whole stack of 8 cubes, $M\in C_8$ and $M$ involves precisely the vertices of $\Delta$.

\noindent (3)
Finally, assume that only three of the four cubes are halved
at least once in direction $c$. Then we may assume that the
rooted tree representing the halving operations done on
$m$, begins exactly as the second tree in Figure 16 below. Note that
at this point, and as a consequence of the geometric
interpretation, we know that $\Delta$ is a subgraph of the
open book $B$  containing the three edges labeled $c$. Also, $B$ lies inside the 8 cube stack associated to $\Delta$. Furthermore, the elements of $B$ correspond to elements of $Y$.
We shall show that $\Delta$ is exactly the open book $B$. Since $\Delta$ is connected it suffices to show  that any two neighbouring cubes in the open book $B$
can be contracted in $Y$. Consider the admissible set $M_a$ obtained as follows: First, halve $m$ in direction $a$ and assume that the second half of $m$, i.e. $m \alpha_a^2$,
contains the only one of the cubes not cut in direction
$c$. Then perform in $m \alpha_a^2$ all halvings needed to
reach those elements of $Y$ that are descendents of $m \alpha_a^2$. The first half of $m$, $m \alpha_a^1$, is not cut anymore. Then $M\leq
M_a$, $M_a=\{m\alpha_a^1\}\cup(M_a\cap Y)$.

Observe that, in the first half of $m$, there are only two colours in the path needed
to obtain the elements of $\Delta\cap\Gamma_{M_a}$ from
$M_a$. As $\Delta \cap \Gamma_{M_a}$ is $*$-connected in $\Gamma_{M_a}$, we may
apply Lemma \ref{tec2} and deduce that the square of the
open book $B$
with edges labeled by $b$ and $c$ is in $\Delta$. The same
argument with $b$ substituted by $a$ implies that
the square of the open book $B$ with edges labeled by $c$ and
$a$ is in $\Delta$. Thus
 $\Delta$ is  the open book $B$.

\bigskip

\begin{tikzpicture}[scale=0.9]

  \draw[black] (0,0) -- (0.5,1) --(1,0);
  \draw[black] (1.5,0) -- (2,1) --(2.5,0);
  \draw[black] (3,0) -- (3.5,1) --(4,0);
   \draw[black] (4.5,0) -- (5,1) --(5.5,0);

  \draw[black,dashed] (0.5,1) -- (1.25,2) --(2,1);
   \draw[black,dashed] (3.5,1) -- (4.25,2) --(5,1);

  \draw[black,dotted] (1.25,2) --(2.75,3) --(4.25,2);

   \draw[black] (8,0) -- (8.5,1) --(9,0);
  \draw[black] (9.5,0) -- (10,1) --(10.5,0);
  \draw[black] (11,0) -- (11.5,1) --(12,0);

  \draw[black,dashed] (8.5,1) -- (9.25,2) --(10,1);
   \draw[black,dashed] (11.5,1) -- (12.25,2) --(13,1);

  \draw[black,dotted] (9.25,2) --(10.75,3) --(12.25,2);

\end{tikzpicture}

\bigskip

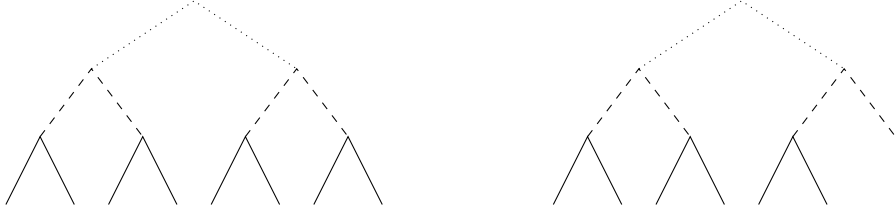
\centerline{Figure 16} 

\bigskip

\end{proof}

\bigskip\noindent We are now ready to prove the analogue to Theorem \ref{pushing} with $M \in C_{8t}$.

\begin{theorem} \label{modified3} Let $s= 3.$ There exists an order reversing poset map
$$M:\{\text{Poset of simplices of }|K_Y|\}\to K_Y$$
such that for any $t$-simplex $\sigma:A_t<A_{t-1}<\ldots<A_0$ we have
$$A_t\leq M(\sigma)\in C_{8t}.$$
\end{theorem}

\begin{proof}
We split the proof into three steps.  Fix a  linear ordering on the colours $a,b,c$.

\medskip\noindent
(1) The definition of $M$ on vertices of $K_Y$. For each admissible $A<Y$ we define a designated edge $M(A)$ as follows:

\noindent
Consider the graph
  $\Gamma_A$. We define $M(A)$ as an edge of $\Gamma_A$ such that if $\Gamma_A = \Gamma_B$ for some $B < Y$, then
  $M(A) = M(B)$. If there is some open book between the $*$-connected components of $\Gamma_A$, we define
  $M(A)$ to be the middle edge of the open book with middle edge of smallest possible colour amongst those open books which are $*$-connected components of $\Gamma_A$.

\medskip
\setlength{\unitlength}{0.7cm}
\begin{picture}(22,4)(-0.1,-0.5)

\put(7,0){\line(0,1){2}}
\put(6.5,1){$c$}
\put(11,0){\line(0,1){2}}
\put(10.5,1){$c$}
\put(7,2){\line(1,0){4}}
\put(7,0){\line(1,0){4}}
\put(9,0){\line(0,1){2}}
\put(7.5,0.8){$M(A)$}
\put(7.8,1.5){$a$}
\put(9.8,1.5){$b$}
\put(7.8,-0.5){$a$}
\put(9.8,-0.5){$b$}

\end{picture}

\medskip

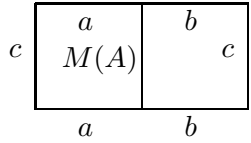
\centerline{Figure 17: The open book extended}

\bigskip

\noindent
If $\Gamma_A$ does not have an open book as a  $*$-connected component, but contains a
$*$-connected component, which is a separate edge $e$, i.e. case 1
of Figure 15, we define $M(A) = e$. Again, there might be more than
one such edge $e$ and we choose  $e$  of smallest
possible colour.

\noindent
If $\Gamma_A$ does not
contain $*$-connected components, which are open books or
separate edges,
we choose $M(A)$
to be an edge of the smallest possible colour of  a $*$-connected component of $\Gamma_A$.

\medskip\noindent
From now on we write $\Delta_A$ for the $*$-connected
component of $\Gamma_A$ such that $M(A) \in \Delta_A$.
We can further assume that if $\Delta_A = \Delta_B$ then
$M(A) = M(B)$.
\medskip

\noindent
(2) Let $A=A_r<A_{r-1}<\ldots<A_0$ be contractions of $Y$ such that all $M(A_i)$ belong to  $\Delta_{A}$. Recall that  each $M(A_i)$ is a simple contraction of $Y$.
Let $\Omega=\{M(A_r),\ldots,M(A_0)\}$ and put $N=\text{glb}_{A}(\Omega)$. We aim to show that $N\in C_8$ and that the vertices of $Y$ involved in $N$ are inside $\Delta_A$.

\noindent
Observe that $\Delta_A$ is $*$-connected. So it must be one of the graphs of Figure 15. If it is an edge, a square or a cube then our claim that $N \in C_8$ follows from Lemma \ref{graphs}.
So we may assume that $\Delta_A$ is an open book. We have
$$\Delta_A=\Delta_A\cap\Gamma_{A_r}\supseteq\ldots\supseteq\Delta_A\cap \Gamma_{A_0}.$$

\noindent The definition of $M$ yields that   if $\Delta_A=\Delta_A\cap\Gamma_{A_r}=\ldots=\Delta_A\cap \Gamma_{A_0}$ then $M(A_r)=\ldots=M(A_0)$. In this case $N=M(A_r)\in C_2$. So we may
 assume that there is some $0\leq i< r$ such that
 $$\Delta_A=\Delta_A\cap\Gamma_{A_r}=\ldots=\Delta_A\cap\Gamma_{A_{i+1}}
 \supsetneq
 \Delta_A\cap \Gamma_{A_i}.$$

\noindent Denote $B=A_i$. We
have $$\Delta_B\subseteq\Delta_A
\cap\Gamma_B\subsetneq \Delta_A.$$

\noindent Moreover, by the definition of $M$, $M(A)=M(A_r)=\ldots=M(A_{i+1})$  is the middle edge
of the open book $\Delta_A$.

\noindent
We claim that $\Delta_A \cap
\Gamma_B$ is a subgraph of one of the following two graphs:

\bigskip
\setlength{\unitlength}{0.7cm}
\begin{picture}(11,6)(-1,0)

\put(7,0){\line(0,1){2}}
\put(6.5,1){$c$}
\put(11,0){\line(0,1){2}}
\put(10.5,1){$c$}
\put(7,2){\line(1,0){2}}
\put(7,0){\line(1,0){2}}
\put(9,0){\line(0,1){2}}
\put(8.5,1){$c$}
\put(7.8,1.5){$a$}
\put(7.8,-0.5){$a$}
\put(4,1){$\Gamma_2$ :}

\put(7,3){\line(0,1){2}}
\put(6.5,4){$c$}
\put(11,3){\line(0,1){2}}
\put(10.5,4){$c$}
\put(9,5){\line(1,0){2}}
\put(9,3){\line(1,0){2}}
\put(9,3){\line(0,1){2}}
\put(8.5,4){$c$}
\put(9.8,4.5){$b$}
\put(9.8,2.5){$b$}
\put(4,4){$\Gamma_1 $ :}

\end{picture}
\bigskip

\centerline{Figure 18}
\goodbreak

\bigskip
\noindent
Observe that $\Delta_A\cap\Gamma_B$ is not connected. Indeed,
in the process of obtaining $B$ from $A$ there was a cutting of a cube containing $U(\Delta_A)$ which halved $U(\Delta_A)$. The  structure of $\Delta_A$ as an open book with three parallel edges $c$ implies that such  a halving cannot be in direction $c$.  The case when the direction of this halving is $a$ corresponds to $\Gamma_1$ (see Fig. 18), i.e. $\Delta_A \cap \Gamma_B \subseteq \Gamma_1$ and the case when the direction is $b$ corresponds to $\Gamma_2$, i.e. $\Delta_A \cap \Gamma_B \subseteq \Gamma_2$.
Alternatively, consider the second tree in Figure 16. 
The commutativity relations do not allow us to move $c$ to the top, whereas having $a$ or $b$ at the top yields a disconnected graph.
A similar argument shows that there is an expansion  $A_{i+1} <  \widetilde{B}$ such that $\Delta_A \cap \Gamma_{\widetilde{B}} = \Gamma_k$, where we have fixed one $k \in \{ 1, 2 \}$ such that $\Delta_A \cap \Gamma_B \subseteq \Gamma_k$.

For any $0\leq j\leq i$ we also have $M(A_j)\in\Delta_{A_j}\subseteq\Delta_A\cap\Gamma_B$.
Then since $\Delta_A \cap \Gamma_B \subseteq \Gamma_k$ we have  $\Omega \subset (\Delta_A \cap \Gamma_B)  \cup \{ M(A) \} \subseteq  \Gamma_k = \Delta_A \cap \Gamma_{\widetilde{B}} \subseteq \Gamma_{\widetilde{B}}$. Hence $A < \widetilde{B} \leq \Omega$ and so
$$\text{glb}_{\widetilde{B}}(\Gamma_k)\leq \text{glb}_{\widetilde{B}}(\Omega)=N.$$

\noindent
Now split $\Gamma_k = D_1 \cup D_2$ into its connected
components, where $D_1$ is the edge and $D_2$ is the
square. Note that $D_1$ and $D_2$ are $*$-connected
components of $\Gamma_{\widetilde{B}}$, hence  Lemma
\ref{graphs}  implies thet $glb_{\widetilde{B}}(D_i)$
involves (i.e. contracts)  $2^i$ vertices ( i.e. elements)  of $Y$. Then by Lemma \ref{remark2}
$
glb_{\widetilde{B}}(D_1 \cup D_2)
$ contracts $2 + 4 = 6$ vertices of $Y$. Hence $N \in C_6 \subseteq C_8$.

\goodbreak

\medskip\noindent
(3)  The definition of $M$ on a simplex of $K_Y$:

\noindent
Let $\sigma : A_t < A_{t-1} < \ldots < A_0$ be a
simplex of $K_Y$ and $t \geq 1$.
Thus
$\Gamma_{A_0} \leq \ldots \leq \Gamma_{A_{t-1}} \leq
\Gamma_{A_t}$ and we have already defined  $M(A_i)$ as an
edge of $\Gamma_{A_i}$ for all $i$.
Let $\Omega = \{ M(A_t),
M(A_{t-1}), \ldots, M(A_0) \}$, which is a
set of
edges of  $\Gamma_{A_t}$.

\noindent Consider the following partition of $\Omega$:

\noindent Put $\alpha_1=t$ and
$$\Omega_1=\Omega\cap\Delta_{A_{\alpha_1}}.$$

\noindent Assume $\Omega_{r-1}$ is defined. If $\bigcup_{i=1}^{r-1}\Omega_i\neq\Omega$,  choose the largest $j\in \{0,...,t\} $ such that
$$M(A_j)\in\Omega\setminus(\bigcup_{i=1}^{r-1}\Omega_i).$$
Rename $A_j$ to $A_{\alpha_r}$
and put $\Omega_r=\Omega\cap\Delta_{A_{\alpha_r}}$. Hence at
each step we have a subchain (i.e. subsimplex) of $\sigma$ satisfying the conditions of  (2).

At some point we will have $\Omega=\bigcup_{i=1}^k \Omega_i$. Let
$$N_i:=\text{glb}_{A_{\alpha_i}}(\Omega_i).$$

By step (2), $N_i\in C_8$ and the vertices of $Y$ involved
in $N_i$ are contained in $\Delta_{A_{\alpha_i}}$. Now we
claim  that these $N_i$ are pairwise disjoint contractions
of $Y$. To see this, let $i\neq j$. We may assume that
$A_{\alpha_i}\leq A_{\alpha_j}$ and therefore
$\Gamma_{A_{\alpha_i}}\supseteq\Gamma_{A_{\alpha_j}}$. As
$\Delta_{A_{\alpha_i}}$ is a $*$-connected component in $\Gamma_{A_{\alpha_i}}$, we deduce that either $\Delta_{A_{\alpha_i}}$ and $\Delta_{A_{\alpha_j}}$ are disjoint (and in this case $N_i$ and $N_j$ are also disjoint) or $\Delta_{A_{\alpha_j}}\subseteq\Delta_{A_{\alpha_i}}$. But the second case is impossible by the construction of the partition above.

\noindent Next we define
$$M(\sigma) = \text{glb}_A(\Omega).$$

\noindent Clearly,
$$M(\sigma)=\text{glb}_A (\{ N_1, \ldots, N_k\}) $$
and, if $k \leq t$, then
$$
M(\sigma) \in C_{8k} \subseteq C_{8t}.
$$
Finally, if $k = t+1$ then all $\Omega_i$
contain precisely one edge, so for all $i$ we have $N_i =
M(A_i)$ and so $M(\sigma)  \in C_{2(t+1)} \subseteq C_{8t}.$

\end{proof}

\noindent
As a corollary we get
  the following modified version of  Proposition
  \ref{oxford1}.

\begin{corollary} \label{modified30}  For any $t$ there is a map
$$f_t:|K_Y|^t\to |K_Y|$$
which is homotopy equivalent to the inclusion $i_t:|K_Y|^t\to|K_Y|$
 such that $f_t(\sigma)\subseteq\Sigma^t_{8t}$.
\end{corollary}

\noindent From now on we can proceed analogously to the case $s=2$. As a first step we have a three-dimensional analogue to Theorem \ref{important31}:

\begin{corollary} \label{important-135} Let $s=3.$ There exists a function $\alpha(t)$ such that if $|Y|\geq\alpha(t),$  the inclusion of  $|K_Y|^t$ in $|K_Y|$ is null-homotopic.
\end{corollary}

\begin{proof} Follow the proofs of Theorem \ref{important31} and Lemma \ref{important32}  substituting Proposition  \ref{oxford1} with Corollary \ref{modified30}.
\end{proof}

\begin{theorem} \label{s=3} The Brin group $3V$ on 3 colours  of arity 2 is of type $\F_\infty$.
\end{theorem}

\begin{proof} The proof follows the proof of Theorem
  \ref{main2}. The main point is the construction of the poset map $M$ of
  Theorem \ref{modified3}. Applying Corollary \ref{important-135}, the
  rest follows as before.
\end{proof}

\end{document}